\documentclass[11pt]{amsart}
\usepackage{amsmath,amsthm,amssymb,amsfonts}
\usepackage{color}
\usepackage{geometry} 

\usepackage[latin1]{inputenc}
\DeclareMathAlphabet{\mathpzc}{OT1}{pzc}{m}{it}

\usepackage{hyperref}

\usepackage{enumitem}

\usepackage{pgf,tikz}
\usetikzlibrary{trees, matrix, arrows, cd}


\makeindex

\newcommand{\ad}{\mathrm{ad} \,}

\newcommand{\SO}{\mathrm{SO}}

\newcounter{consta}
\renewcommand{\theconsta}{{\kappa_{\arabic{consta}}}}
\newcounter{constb}[section]

\newcounter{constc}[section]

\newcommand{\consta}{\refstepcounter{consta}\theconsta}

\newcommand{\Sob}{\mathcal{S}}

\newcommand{\Ad}{\mathrm{Ad}}

\newcommand{\Lie}{\mathrm{Lie}}

\newcommand{\vol}{\operatorname{vol}}

\newcommand{\G}{\mathbf{G}}

\newcommand{\height}{\operatorname{ht}}

\newcommand{\SL}{\mathrm{SL}}

\newcommand{\R}{\mathbb{R}}

\newcommand{\Z}{\mathbb{Z}}
\newcommand{\Q}{\mathbb{Q}}

\newcommand{\adele}{\mathbb{A}}

\DeclareMathAlphabet{\mathpzc}{OT1}{pzc}{m}{it}
\DeclareFontFamily{OT1}{rsfs}{}
\DeclareFontShape{OT1}{rsfs}{n}{it}{<-> rsfs10}{}
\DeclareMathAlphabet{\mathscr}{OT1}{rsfs}{n}{it}

\numberwithin{equation}{section}

\swapnumbers
\newtheorem{thm}[subsection]{Theorem}

\newtheorem*{cor*}{Corollary}
\newtheorem{lemma}[subsection]{Lemma}
\newtheorem*{lem}{Lemma}
\newtheorem{propo}[subsection]{Proposition}

\newtheorem{thm*}{Theorem}[]
\newtheorem*{prop}{Proposition}

\newtheorem*{claim}{Claim}
\newcommand{\bbr}{\mathbb{R}}
\newcommand{\bbq}{\mathbb{Q}}
\newcommand{\bbz}{\mathbb{Z}}

\newcommand{\Ff}{\mathbb F}

\newcommand{\Bcal}{{\mathcal B}}

\newcommand{\Gfrak}{\mathfrak{G}}

\newcommand{\hfrak}{\mathfrak{h}}
\newcommand{\rfrak}{\mathfrak{r}}
\newcommand{\gfrak}{\mathfrak{g}}

\newcommand{\places}{\Sigma}

\newcommand{\red}{{\rm red}}
\newcommand{\be}{\begin{equation}}
\newcommand{\ee}{\end{equation}}

\newcommand{\temp}{{M}}

\newcommand{\disc}{\operatorname{ht}}

\newcommand{\slnr}{\SL_N(\adele)}

\newcommand{\rad}{{\rm R}}
\newcommand{\gsl}{\mathfrak{sl}}
\newcommand{\injr}{\eta}

\newcommand{\cfun}{\mathsf{c}}
\newcommand{\cfs}{\cfun_S}
\newcommand{\cfb}{\cfun_{\mathcal B}}

\newcommand{\lcagr}{{\bf L}}
\newcommand{\lcgr}{L}
\newcommand{\tlg}{{\tilde{\lcagr}}}
\newcommand{\tlc}{{\tilde{\lcgr}}}
\newcommand{\hl}{\height_\lcgr}

\newcommand \A {\mathbb{A}}%

\newcommand \bG {\mathbf{G}}%
\newcommand \bH {\mathbf{H}}%
\newcommand \bL {\mathbf{L}}%
\newcommand \bR {\mathbf{R}}%

\newcommand \fg {\mathfrak{g}}%
\newcommand \fl {\mathfrak{l}}%
\newcommand \fh {\mathfrak{h}}%
\newcommand \fo {\mathfrak{o}}%
\newcommand \fr {\mathfrak{r}}%

\newcommand \rH {\mathrm{H}}%
\newcommand \rN {\operatorname{N}}%
\newcommand \fsl {\operatorname{\mathfrak{sl}}}%

\begin{document}

\title[Diameter of homogeneous spaces]{Diameter of homogeneous spaces: an effective account}

\author{A. Mohammadi}
\address{A.M.:  Department of Mathematics, The University of California, San Diego, CA 92093, USA}
\email{ammohammadi@ucsd.edu}
\thanks{A.M.~acknowledges support from the NSF and Alfred P.\ Sloan Research Fellowship.}

\author{A. Salehi Golsefidy}
\address{A.S-G: Department of Mathematics, The University of California, San Diego, CA 92093, USA}
\email{golsefidy@ucsd.edu}
\thanks{A.S-G.~acknowledges support from the NSF and Alfred P. Sloan Research Fellowship.}

\author{F. Thilmany}
\address{F.T.: Department of Mathematics, The University of California, San Diego, CA 92093, USA}
\email{fthilman@ucsd.edu}
\thanks{F.T.~acknowledges support from the Fonds National de la Recherche, Luxembourg.}

\begin{abstract}
 
In this paper we prove explicit estimates for the size of small lifts of points in homogeneous spaces.
Our estimates are polynomially effective in the {\em volume} of the space and the injectivity radius. 
\end{abstract}

\maketitle

\section{Introduction} \label{sec:intro}
Let $G$ be a semisimple Lie group and let $\Gamma\subset G$ be an arithmetic lattice, e.g.~$G=\SL_d(\R)$ and $\Gamma=\SL_d(\Z)$.  
Reduction theory provides a description of a (weak) fundamental domain for $\Gamma$ in $G$.
Among other things, it relates the injectivity radius at a point $x\in G/\Gamma$ to the size of a {\em small} lift for $x$ in $G$. 
In general, however, these estimates are only up to a compact subset of $G$; in particular, when $\Gamma$ is a uniform (cocompact) lattice in $G$ one does not obtain explicit estimates on the diameter of $G/\Gamma$.

In this paper we provide an explicit estimate for the size of a {\em small} lift in $G$ of a point $x\in G/\Gamma$; 
our estimates are polynomial in the injectivity radius at $x$  
and in a certain measure of the arithmetic complexity of $\Gamma$ which is closely related to the volume of $G/\Gamma$, see Theorem~\ref{thm:adelic-red}. 

It is plausible that some of the arguments involved in reduction theory can be effectivized; this paper however takes an alternative route. The proofs here rely on a uniform spectral gap for arithmetic quotients in the case of semisimple group; see e.g.~\cite{Grom, BurSc, LiMarg} for a similar approach. We then prove and utilize an effective Levi decomposition, in \S\ref{sec:small-levi} and \S\ref{sec:eff-levi}, to allow for groups which may not be semisimple. 

It is worth mentioning that when $\Gamma$ is a cocompact lattice, the dependence of our estimates 
on the injectivity radius may be omitted, see \S\ref{sec:Mahler-Meas}. 
The reader may compare this to the analysis in~\cite{BurSc}, where similar estimates for the isometry groups of rank one symmetric spaces are proved. 
However, our multiplication constants are allowed to depend on the number $N$ which is defined in \S\ref{sec:data} --- this number can be thought of as a notion of dimension for the arithmetic datum that defines $\Gamma$.

The main results are first formulated and proved (in \S\ref{sec:proof}) in the adelic language. Then we deduce the results for the $S$-arithmetic case --- in particular for the case of semisimple Lie groups --- from the adelic setting. In addition to providing a uniform treatment, the adelic language has the advantage that we may bring to bear the seminal works of Prasad~\cite{Pr} and Borel and Prasad~\cite{BorPr}, \`{a} la~\cite{EMMV}, to avoid assuming any {\em splitting} conditions in Theorem~\ref{thm:adelic-red}. In \S\ref{sec:local}, we discuss some corollaries of this theorem in the $S$-arithmetic setting; see namely Theorem~\ref{thm:local-eff-diam-intro} and the discussion following it.

\subsection{The notion of an algebraic datum}\label{sec:data}\label{sec:homog-meas}\label{sec:statement}
In the following,~$\adele$ denotes
the ring of adeles over~$\bbq$. We let $\places=\{\infty\}\cup\{p: p \textrm{ is a prime}\}$ denote the set of places of $\bbq$, and let
$\places_f$ be the set of finite places. We sometime write $\places_\infty$ for the set containing the infinite place.
We will denote places in $\places$ by $v, w, ...$ and places in $\places_f$ by $p, q, ...$.
In this notation, we often write $\bbq_v$ to denote $\bbr$ or $\bbq_p$. 

Throughout, we assume fixed the following datum $(\G, \iota)$:  
\begin{enumerate}
\item A connected algebraic $\bbq$-group $\G$ whose solvable radical is unipotent, i.e.~$\rad(\G)=\rad_u(\G)$.
\item We will always assume $\G$ to be simply connected.
\item An algebraic homomorphism $\iota: \G \rightarrow \SL_N$ 
defined over $\bbq$, with a central kernel.
\end{enumerate}
Condition~(1) is equivalent to $\operatorname{Hom}(\G,\G_m)=\{1\}$. In particular, we get that
$\operatorname{Hom}_\bbq(\G,\G_m)=\{1\}$, hence $\G(\adele)/\G(\bbq)$ has a $\G(\adele)$-invariant finite measure.

Set $X=\SL_N(\adele)/\SL_N(\bbq)$; we let $\vol_X$ denote the $\slnr$-invariant probability measure on $X$. 
Let $G=\iota(\G(\adele))$ and $Y =\iota(\G(\adele)/\G(F)) \subset X$. Let $\mu_Y$ (or simply $\mu$ when there is no confusion) 
be a $G$-invariant probability measure on $Y$.
Let $m$ be a Haar measure on $G$ which projects to $\mu$ under the orbit map.

\subsection{A height function on $X$}\label{sec:alpha1}

For any $v \in \places$, we will abusively let $\|\;\|_v$ denote the maximum norm (with respect to the standard basis) both on $\Q_v^N$ and on $\fsl_N(\Q_v)$.
For any $w \in \A^N$, we set
\[
\cfun(w) :=\prod_{v \in \places} \|w_v \|_v.
\] 
Thanks to the product formula, we have $\cfun(rw)=\cfun(w)$ for all $r\in\bbq$, $w\in \A^N$. 
Moreover, for all $w \in \Q^N - \{0\}$, $\cfun(w)$ is an integer and $\cfun(w) \geq 1$.

We define the \emph{height} function $\height: \SL_N(\A) \to\bbr^+$ by
\be \label{eq:def-ht-g}
\height(g):=\max\{\cfun(gw)^{-1}: 0\neq w \in \Q^N\}.
\ee
This function is $\SL_N(\Q)$-invariant, hence induces a function on $X$ which we continue to denote by $\height$. That is:
for any $x\in X$ we put $\height(x)=\height(g)$ where $g\in \SL_N(\A)$ is so that $x=g\SL_N(\Q)$. 

For every $p\in\places_f$ we let $\|\;\|_{{\rm op},v}$ (or simply $\|\;\|_{\rm op}$ when there is no confusion) denote the operator norm on $\SL_N(\bbq_v)$, induced using the norm $\| \; \|_v$ on $\bbq_v^N$. For any $g\in\SL_N(\bbq_v)$ define
\[
|g| :=\max\{\|g\|_{\rm op},\|g^{-1}\|_{\rm op}\}. 
\]

\subsection{Complexity of homogeneous sets}\label{sec:volume}
An intrinsic notion of {\em volume of the datum} $(\G,\iota)$ was defined and utilized in~\cite{EMMV}; we recall the definition here.

Fix an open subset $\Omega \subset \slnr$ that contains the identity and has compact closure (see \S \ref{sec:X-eta} for our choice for $\Omega$).
Set 
\begin{equation} \label{eq:volume}
\vol(Y) := m (G \cap \Omega)^{-1}.
\end{equation}

Evidently this notion depends on $\Omega$, but the notions arising from two different choices of $\Omega$
are comparable to each other, in the sense that their ratio is bounded above and below.
Consequently, we drop the dependence on $\Omega$ in the notation.  
See~\cite[\S2.3]{EMMV} for a discussion of basic properties of the above definition.

\subsection{Height of rational subspaces}\label{sec:height}

Let ${\bf W}\subset\fsl_N(\Q)$ be a $d$-dimensional subspace,
so $\wedge^{d}{\bf W}$ is a rational line in $\wedge^{d}\fsl_N(\Q)$. 
This line is diagonally embedded in $\wedge^{d}\fsl_N(\A),$ and we do not 
distinguish between this diagonal embedding and the line. 

We endow $\wedge^{d} \fsl_N(\Q_v)$ with the maximum norm with respect to the basis obtained by collecting the $d$-fold wedges of (distinct, ordered) elements of the canonical basis of $\fsl_N(\Z)$. In this section, we will again use $\| \; \|_v$ to denote this norm. 

Let ${\bf v}_{\bf W}$ denote a primitive integral vector on $\wedge^{d}{\bf W}$ --- 
this vector is obtained by fixing a $\Z$-basis for ${\bf W}\cap\fsl_N(\Z)$. 
Define
\be\label{eq:def-ht}
\height({\bf W}):=\|{\bf v}_{\bf W}\|_\infty.
\ee
This is independent of the choice of the basis;
moreover, because we used the max norm in the above definition, $\height({\bf W})$
is an integer.
Alternatively, $\height({\bf W})$ may be defined as follows. Let $\{{e}_1,\ldots,{e}_{d}\}$
be a $\bbq$-basis for ${\bf W}.$ Then
\[
\height({\bf W})=\prod_{v}\|{e}_1\wedge\cdots\wedge{e}_{d}\|_{v}
\]
where the product is taken over all places of $\bbq.$ In view of the product formula, 
the above is independent of our choice of the rational basis for ${\bf W}.$  

Given a $\bbq$-subgroup $\bH$ of $\SL_N$ we define
\be\label{eq:def-ht-group}
\height(\bH):=\height\bigl(\Lie(\bH)\bigr)=\|{\bf v}_\bH\|_\infty,
\ee
where ${\bf v}_{\bH}$ is a primitive integral vector as above. If $\bH$ is a $\bbq$-subgroup of $\bG$ instead, we set $\height(\bH) =\height(\iota(\bH))$. 

The volume of an adelic orbit defined in \S \ref{sec:volume} is closely related to the height function. 
This relationship is easy to describe for unipotent groups and was studied in~\cite[App.~B]{EMMV}, under the assumption that $\G$
is semisimple.

We now define the height of $Y$ to be
\be \label{eq:def-discY}
\disc(Y):=\max\{\height(\G), \vol(Y)\}.
\ee

The following theorem is the main result of this paper.

\begin{thm}\label{thm:adelic-red}
There exists some $\consta\label{k:main-exp-arch-norm}\label{k:main-exp-inj-r}\label{k:main-exp-vol}>0$ depending only on $N$, and for any datum $(\G,\iota)$ as in \S\ref{sec:data} there exists some  
$p\in\places_f$ with 
\[
p\ll\bigl(\log\disc(Y)\bigr)^2,
\] 
so that the following holds. For each $g \in\G(\A)$, there exists some $\gamma\in\G(\Q)$ such that $\iota(g\gamma)_{q} \in \SL_N(\Z_q)$ for all primes $q \neq p$,
\[
\text{$|\iota(g\gamma)_{\infty}| \ll \bigl(\log\disc(Y)\bigr)^{\ref{k:main-exp-arch-norm}}\;\;$ and $\;\;|\iota(g\gamma)_p| \ll {\height(\iota(g))}^{\ref{k:main-exp-inj-r}}\disc(Y)^{\ref{k:main-exp-vol}}$.}
\]
Moreover, the implicit multiplicative constants depend only on $N$. 
\end{thm}

The existence of such a prime $p$ relies on Prasad's volume formula~\cite{Pr}, see \S\ref{sec:semisimple}
for more details. 

\subsection{The $S$-arithmetic setting}\label{sec:S-arith-intro} 
Let $S\subset\places$ be a finite subset which contains the infinite place. We will write $\Q_S$ for $\prod_{v\in S}\Q_v$, and $\Z_S$ will denote the ring of $S$-integers. 

Define $\height_S: \SL_N(\Q_S) \to\bbr^+$ by
\[
\height_S(g):=\max \left\{(\textstyle\prod_{v\in S}\|gw\|_v )^{-1}: 0\neq w \in \Z_S^N \right\}.
\]

For any $S$ as above, define $\Delta_S$ (or simply $\Delta$ if there is no confusion) by
\[
\Delta_S :=\text{the projection of $\G(\Q)\cap\Bigl(\G(\Q_S) \times \textstyle\prod_{q\not\in S}\iota^{-1}(\SL_N(\Z_q)\Bigr)$ to $\G(\Q_S)$};
\] 
note that $\Delta_S$ is a lattice in $\G(\Q_S)$.

Let $\Omega_S\subset\SL_N(\Q_S)$ be an open set which contains the identity and has compact closure. Put $\hat Y=\iota(\G(\Q_S)/\Delta)$ and define
\[
\vol(\hat Y) :=m_S\bigl(\iota(\G(\Q_S))\cap\Omega_S\bigr)^{-1},
\] 
where $m_S$ is a Haar measure on $\iota(\G(\Q_S))$
normalized so that $m_S(\hat Y)=1$.

\begin{thm}\label{thm:local-eff-diam-intro} 
Let $(\bG, \iota)$ be as in \S\ref{sec:data}. 
Let $S$ be a finite set of places of $\Q$ which contains the infinite place. 
For every $v\in S$, let $\G_v$ be a semisimple algebraic $\Q_v$-group.
Assume
\begin{enumerate}
\item $\G_v$ and $\G$ are isomorphic over $\Q_v$; in particular, $\G$ is semisimple and $\G_v$ is simply connected.   
\item The group $\G(\Q_S)=\prod_{v\in S}\G(\Q_v)=\prod_{v\in S}\G_v(\Q_v)$ is not compact.
\end{enumerate}

There exists a constant $\consta\label{k:local-eff-diam-intro-1}\label{k:local-eff-diam-intro-2}>0$ depending only on $N$ 
and a constant $C\geq1$ which depends on $\G(\Q_S)$ and $N$, but not on $\G$, so that the following holds.  
For every $g\in\G(\Q_S)$ there exists some $\delta\in\Delta$ such that
\[
|\iota(g\delta)|\leq C\,{\height_S(\iota(g))}^{\ref{k:local-eff-diam-intro-1}}\vol(\hat Y)^{\ref{k:local-eff-diam-intro-2}}.
\]
\end{thm}

This theorem will be proved in~\S\ref{sec:local}; see in particular Theorem~\ref{thm:eff-diam-local-semisimple} where Theorem~\ref{thm:local-eff-diam-intro} is restated and proved. We will also discuss some other corollaries of Theorem~\ref{thm:adelic-red} in~\S\ref{sec:local}.

Let us highlight two features of the above theorem.
First, note that once $N$ is fixed the dependence on the lattice $\Delta$ in the estimates is only through its \emph{covolume} $\vol(\hat Y)$.
Second, the above estimates use $\vol(\hat Y)$ instead of $\vol(Y)$; the fact that $\vol(\hat Y)$ and $\vol(Y)$ are polynomially related to each other is a consequence of deep results by Prasad and Borel and Prasad~\cite{Pr,BorPr}, see~\S\ref{prop:vol-vol}.


\section{Notation and preliminaries}

\subsection{Notation}\label{sec:notation}

Throughout the paper, $\places$, $\adele$, etc.\ will be as in \S\ref{sec:statement}. 
In particular, $\adele=\prod_{v\in\places}'\bbq_v$ where~$\prod'$ denotes
the restricted direct product with respect to $\bbz_p$ for~$p\in \places_f$.
Given an element $g$ in $\SL_N(\A)$ (or in $\fsl_N(\A)$, $\A^N$, etc.), we write $g_v$ for the $v$-th component of $g$.

If $S \subset \Sigma$ is a finite set of places containing the infinite place, $\Z_S$ will denote the ring of $S$-integers, that is $\Z_S = \{r \in \Q \mid |r|_v \leq 1 \text{ for } v \notin S \}$. On the other hand, $\Q_S$ will denote the product $\prod_{v \in S} \Q_v$. There are canonical inclusions $\Q \subset \A$, $\Q \subset \Q_S$, $\Q_S \subset \A$, etc.~which will often be omitted from the notation. 

For any finite place~$p\in\places_f$, $\Ff_p=\bbz_p/p\bbz_p$ is the finite field of order $p$. Let $|x|_p$ denote the absolute value on $\bbq_p$ normalized so that $|p|_p=1/p$. 
Finally, let $\widehat{\bbq_p}$ denote the maximal unramified extension of $\bbq_p$, $\widehat{\Z_p}$ denote the ring of integers in $\widehat{\bbq_p}$, and $\widehat{\Ff_p}$ denote the residue field of $\widehat{\bbz_p}$. Note that $\widehat{\Ff_p}$ is the algebraic closure of $\Ff_p.$ 

Recall that, for any place $v \in \places$, $\|\;\|_v$ denotes the maximum norm both on $\Q_v^N$ and $\gsl_N(\bbq_v)$ with respect to their standard bases. When there is no ambiguity, we may drop the subscript $v$. For this norm, we denote $B_{\gsl_N(\bbq_v)}(r)$ the ball in $\gsl_N(\bbq_v)$ of radius $r$ centered at $0$.

Let $(\G,\iota)$ be an algebraic datum, as described in~\ref{sec:statement}.
For any $v\in\places$, let $\gfrak_v =\Lie(\G(\bbq_v))$.  
Using the embedding $d \iota: \gfrak_v \to \fsl_N(\Q_v)$, we pull back the norm $\| \; \|_v$ to a norm on $\gfrak_v$ which we continue to denote by $\| \; \|_v$ (or $\|\;\|_\infty$, $\|\;\|_p$). For these norms, we define $B_{\gfrak_v}(r)$ to be the ball in $\gfrak_v$ of radius $r$ centered at $0$. 

For every $v\in\places$, we let $G_v=\iota(\G(\bbq_v))$; in particular, $G_\infty=\iota(\G(\R))$.

\medskip

In the sequel, the notation $A \ll B$ means: there exists a constant $c>0$ so that $A \leq cB$; the implicit constant $c$ is permitted to depend on
$N$, but (unless otherwise noted) not on anything else. We write~$A\asymp B$ if~$A\ll B\ll A$.
If a constant (implicit or explicit) depends on another parameter or only on $N$, 
we will make this clear by writing e.g.~$\ll_\epsilon$, $\asymp_N$,~$c(G)$, etc.

The exponents $\kappa_\bullet$ are allowed only to depend on $N$. 
We also adopt the $\star$-notation from~\cite{EMV}. 
We write $B = A^{\pm\star}$ if $B=cA^{\pm \kappa},$ 
where $\kappa>0$ and $c$ depend only on $N$, unless it is explicitly mentioned otherwise.
Similarly one defines $B\ll A^\star,$ $B\gg A^\star$.
Finally, we also write~$A\asymp B^\star$ if~$A^\star\ll B\ll A^\star$ (possibly with different exponents).

\subsection{Injectivity radius in $X$}\label{sec:X-eta}

Given $\eta>0$, put $\Xi_{\eta} :=\exp(B_{\gsl_N(\R)}(\eta))$.
Throughout, we assume ${\injr}_0$ is small enough so that $\exp:B_{\gsl_N(\R)}({\injr}_0)\to\Xi_{{\injr}_0}$ is a diffeomorphism.
For any ${\injr}>0$, let 
\[
\Omega_{\injr} :=\Xi_{\injr}\times \Bigl(\prod_{\places_f}\SL_N(\Z_p)\Bigr).
\] 
We fix $\Omega = \Omega_{\eta_0}$; this set will be the one used to measure the volume of $(\bG, \iota)$, as described in \S \ref{sec:volume}.

For $x\in X$, define $\pi_x: \SL_N(\A) \to X$ by $\pi_x(g)=gx$; when $x=e$ we simply write $\pi$ for $\pi_x$.
For every $0<{\injr}<{\injr}_0$, define
\be\label{eq:def-X-eta}
X_{\injr} :=\{x\in X: \pi_x\text{ is injective when restricted to $\Omega_\injr$}\}.
\ee
If ${\injr}\geq {\injr}_0$, set $X_{\injr}=\emptyset$.
Let $Y_{\injr}:= Y\cap X_{\injr}$.

\begin{lemma}\label{lem:ht-injr}
There exists some constant $\consta\label{eq:ht-injr} > 0$ so that the following holds. 
\begin{enumerate}
\item For any $g\in\SL_N(\A)$ we have $g\in X_{\ref{eq:ht-injr}\height(g)^{-\ref{eq:ht-injr}}}$.
\item If $g\SL_N(\Q)\in X_\eta$, then $\height(g)\ll \eta^{-\ref{eq:ht-injr}}$.
\end{enumerate} 
\end{lemma}

\begin{proof}

Let $g\in\SL_N(\A)$. First note that by strong approximation for $\SL_N$, there exists some 
$\gamma_0\in\SL_N(\Q)$ so that 
\[
g\gamma_0=(g_\infty', (g_p'))\in \SL_N(\R)\times\bigl(\textstyle\prod_p\SL_N(\Z_p)\bigr).
\] 
Further, using the reduction theory of $\SL_N(\R)$, there exists some $\hat\gamma_1\in\SL_N(\Z)$ so that
$g_\infty\hat\gamma_1=kau$, where $k\in{\SO}_N(\R)$, $a={\rm diag}(a_i)$ is diagonal with positive entries satisfying
$a_ia_{i+1}^{-1}\leq 2/\sqrt3$, and $u=(u_{ij})$ is unipotent upper triangular with $|u_{ij}|\leq 1/2$. 
Note that 
\[
\|aua^{-1}\|_{\mathrm{op}} \leq \frac{1}{2} \left( \frac{(2/\sqrt3)^{N} - 1}{(2/\sqrt3) - 1} + 1 \right) \ll 1
\]
for any $a$ and $u$ as above.

Let $\gamma=\gamma_0\gamma_1$, where $\gamma_1$ denotes the diagonal embedding on $\hat\gamma_1$ in $\SL_N(\Q)\subset\SL_N(\A)$.
Then, since $\gamma_{1,p}\in\SL_N(\Z_p)$ for all $p$, we have 
\be \label{eq:strg-app}
g\gamma=(kau,(\hat g_p))\in \SL_N(\R)\times\bigl(\textstyle\prod_p\SL_N(\Z_p)\bigr).
\ee

For $w \in \Q^N$, we have
\begin{align}
\notag \cfun(g \gamma w)	&=\| (kau) w\|_\infty \cdot \bigl(\textstyle\prod_{p}\| \hat g_p w\|_p\bigr) \\
\notag					&=\| (kaua^{-1}) a w\|_\infty\cdot \bigl(\textstyle\prod_{p}\|w\|_p\bigr)&& \textrm{since } \hat g_p\in\SL_N(\Z_p) \\
\label{eq:ht-injr-1}			&\asymp \| aw\|_\infty\cdot \bigl(\textstyle\prod_{p}\|w\|_p\bigr)&& \hspace{-4em} \textrm{since } k \in \SO_N(\R), \|aua^{-1}\|_\mathrm{op} \ll 1.
\end{align}
Moreover, we have $\| a w \|_\infty \leq (\max_{i} a_i ) \|w\|_\infty = |a| \cdot \|w\|_\infty$, and thus also $|a|^{-1} \|w\|_\infty \leq \| aw \|_\infty$. 
Therefore,~\eqref{eq:ht-injr-1} implies that 
\be \label{eq:ht-injr-2}
|a|^{-1} \cfun(w)^{-1} \ll \cfun((g\gamma) w)^{-1} \ll |a| \cfun(w)^{-1}.
\ee
Now for $w$ an appropriate basis vector, we have
\[
\| aw\|_\infty^{-1} = (\min_{i} a_i)^{-1} = \max_{i} a_i^{-1} = \| a^{-1} \|_{\rm op}, 
\]
and since $\det a = 1$, we have $\| a^{-1} \|_{\rm op} \geq \| a \|_{\rm op}^{1/(N-1)}$. 
For such $w$, it thus follows from \eqref{eq:ht-injr-1} that $\cfun((g\gamma) w)^{-1} \gg |a|^{1/(N-1)}$. Together with \eqref{eq:ht-injr-2}, this shows
\be\label{eq:ht-injr-3}
|a|^{1/(N-1)} \ll \height(g) = \max\{\cfun(\Ad(g\gamma)w)^{-1}: 0\neq w\in \fsl_N(\Q) \} \ll |a|.
\ee

Now if instead $w \in \fsl_N(\Q)$, we have
\[
\| \Ad(a) w\|_\infty \leq (\max_{i,j} a_i a_j^{-1}) \|w \|_\infty \leq |a|^2 \|w \|_\infty. 
\]
In the same way as above, since $k \in \SO_N(\R)$ and $\|aua^{-1}\|_\mathrm{op} \ll 1$, there is some $c \ll 1$ such that for any $\eta > 0$, 
\[
\Ad(k(aua^{-1}) a)^{-1} B_{\fsl_N(\R)}(\eta) \subset \Ad(a)^{-1} B_{\fsl_N(\R)}(c \eta) \subset B_{\fsl_N(\R)}(c |a|^2 \eta).
\]
Applying the exponential map yields
\[
(k(aua^{-1})a)^{-1}\Xi_{\eta}(k(aua^{-1})a)\subset \Xi_{c |a|^2 \eta}.
\]
Therefore, we have 
\begin{align*}
\gamma^{-1}g^{-1} \Omega_{\eta} g\gamma \cap \SL_N(\Q)&\subset \bigl((kau)^{-1}\Xi_{\eta} kau \cap \SL_N(\Z)\bigr)\times\bigl(\textstyle\prod_p\SL_N(\Z_p)\bigr) \\
&\subset\bigl(\Xi_{c |a|^2 \eta}\cap \SL_N(\Z)\bigr)\times\bigl(\textstyle\prod_p\SL_N(\Z_p)\bigr).
\end{align*}
In particular, if $\eta\ll |a|^{-2}$, then $\gamma^{-1}g^{-1} \Omega_{\eta} g\gamma \cap \SL_N(\Q)=\{1\}$.
That is: $g\in X_{c' |a|^{-2}}$ for perhaps another constant $c' > 0$. This implies the claim in~(1) in view of~\eqref{eq:ht-injr-3}.

To see~(2) in the lemma, let $\eta>0$ and suppose $g\SL_N(\Q)\in X_\eta$.
Let $\gamma\in\SL_N(\Q)$ be so that $g\gamma$ is as in~\eqref{eq:strg-app}. 
For any $w \in \fsl_N(\R)$ in the appropriate root space, we have
\begin{align}
\notag \| \Ad(a)w\|_\infty^{-1} 	&= (\min_{i,j} a_i a_j^{-1})^{-1} \|w\|_\infty = (\max_{i,j} a_i a_j^{-1}) \|w\|_\infty \\
\label{eq:ht-injr-4} 			&\geq N^{-2} \Big(\sum_i a_i \Big) \Big( \sum_j a_{j}^{-1} \Big) \|w\|_\infty \geq N^{-2} |a| \cdot \|w\|_\infty.
\end{align}
Because $k \in \SO_N(\R), \|aua^{-1}\|_\mathrm{op} \ll 1$, we may scale $w$ so that 
\[
w \in \Ad(kaua^{-1})^{-1} B_{\fsl_N(\R)}(\eta)
\] 
while keeping $\|w\|_\infty \gg \eta$. 
With this choice for $w$, we have
\[
\Ad(a)^{-1} w \in \Ad(kau)^{-1} B_{\fsl_N(\R)}(\eta),
\]
that is, $\exp(\Ad(a)^{-1} w) \in (kau)^{-1}\Xi_{\eta} kau$. In consequence, 
\be \label{eq:ht-injr-5}
\| \Ad(a)^{-1} w\|_\infty < 1.
\ee
Indeed, otherwise, we would be able to pick $\Ad(a)^{-1} w$ to be an elementary matrix, for which we would have
\begin{align*}
\exp(\Ad(a)^{-1} w) \in&~\bigl((kau)^{-1}\Xi_{\eta} kau \times \textstyle\prod_p\SL_N(\Z_p)\bigr) \cap \SL_N(\Q) \\
& = \gamma^{-1} g^{-1} \Omega_\eta g \gamma \cap \SL_N(\Q).
\end{align*}
This contradicts the fact $g^{-1} \Omega_\eta g \cap \SL_N(\Q) = \{1\}$. 
In virtue of our choice for $w$, \eqref{eq:ht-injr-4}, and \eqref{eq:ht-injr-5} , we have
\[
|a| \eta \ll |a| \cdot \|w\|_\infty \ll \| \Ad(a)^{-1} w\|_\infty < 1.
\]
Finally, in view of \eqref{eq:ht-injr-3}, this immediately implies
\[
\height(g) \ll |a| \ll \eta^{-1}
\]
and concludes the proof of the lemma. 
\end{proof}

\subsection{Remark} In the definition \eqref{eq:def-ht-g} of the height, instead of using the action of $\SL_N(\A)$ on $\A^N$, one could have acted on $\fsl_N(\A)$ via the adjoint action. More precisely, one could have defined $\widetilde{\height}: \SL_N(\A) \to\bbr^+$ by
\[
\widetilde{\height}(g):=\max\{\cfun(\Ad (g) w)^{-1}: 0\neq w \in \fsl_N(\Q)\},
\]
where the function $\cfun$ is given by the same expression $\cfun(w) :=\prod_{v \in \places} \|w_v \|_v$. 
The proof of lemma \ref{lem:ht-injr} can be used to show that $|a| \ll \widetilde{\height}(g) \ll |a|^2$ (with $a$ as in \eqref{eq:strg-app}), and in consequence that 
\[
\height(g) \ll \widetilde{\height}(g) \ll \height(g)^{2(N-1)}. 
\]
The two heights are thus polynomially related, and for the purpose of Theorem \ref{thm:adelic-red}, they can be used interchangeably.

\subsection{Elements from Bruhat-Tits theory}\label{sec:BrTits} 
We recall a few facts from Bruhat-Tits theory, see~\cite{Ti} and references there for the proofs.
Let $\G$ be a connected semisimple group defined over $\bbq.$ 
Let $p$ be a finite place, then 
\begin{enumerate}
\item For any point $o$ in the Bruhat-Tits building of $\G(\bbq_p),$ 
there exists a smooth affine group scheme $\Gfrak_{p}^{(o)}$ over 
$\Z_p,$ unique up to isomorphism, such that: 
its generic fiber is $\G(\bbq_p),$ and the compact open subgroup $\Gfrak_p^{(o)}(\Z_p)$  
is the stabilizer of $o$ in $\G(\bbq_p),$ see~\cite[3.4.1]{Ti}. \vspace{1mm} 
\item If $\G$ splits over $\Q_p$ and $o$ is a {\it special} point, then\index{special point in the Bruhat-Tits building}
the group scheme $\Gfrak_p^{(o)}$ is a Chevalley group scheme with
generic fiber $\G,$ see~\cite[3.4.2]{Ti}.\vspace{1mm} 
\item $\red_p:\Gfrak_p^{(o)}(\Z_p)\rightarrow\underline{\Gfrak_p}^{(o)}(\Ff_p),$ 
the reduction mod $p$ map, is surjective, see~\cite[3.4.4]{Ti}. \vspace{1mm} 
\item $\underline{\Gfrak_p}^{(o)}$ is connected and semisimple if and only if $o$ is a {\it hyperspecial} point. 
Stabilizers of hyperspecial points in $\G(\bbq_p)$ will be called hyperspecial subgroups, see~\cite[3.8.1]{Ti} and~\cite[2.5]{Pr}.\index{hyperspecial points and subgroups}
\end{enumerate}

If $\G$ is quasi-split over $\bbq_p,$ and splits over $\widehat{\Q_p},$ 
then hyperspecial vertices exists, and they are compact open subgroups of maximal volume. 
Moreover a theorem of Steinberg implies that $\G$ is quasi-split over $\widehat{\Q_p}$ 
for all $p,$ see~\cite[1.10.4]{Ti}.

It is known that for almost all $p$ the group $\G$ 
is quasi-split over $\Q_p,$ see ~\cite[Theorem 6.7]{PR}.
Moreover, for almost all $p$ the groups $K_p$ are hyperspecial, see~\cite[3.9.1]{Ti}.


\section{Small Levi decomposition in Lie algebras}\label{sec:small-levi}
Recall from \S\ref{sec:notation} that $\|\;\|$ 
denotes the (archimedean) max norm both on $\Q^N$ and on $\fsl_N(\Q)$ with respect to the standard basis. 
Note that if $u,v \in \fsl_N(\Q)$, we have $\|[u,v]\| \ll \|u\| \|v\|$. 

If $\fg$ is a subalgebra of $\fsl_N(\Q)$, and $\mathcal B= \{u_1, \dots, u_M \}$ is a 
$\Z$-basis of $\fg \cap \fsl_N(\Z)$, we can also endow $\fg$ with the max norm $\| \; \|_{\mathcal B}$ in the basis $\mathcal B$. 
For any $u \in \fg$ we have
\[
(\max_i \|u_i\|)^{-1} \|u\|\ll \|u\|_{\mathcal B} \ll (\max_i \|u_i\|) \|u\|.
\]

In this section we prove the following.

\begin{propo} \label{LiesmallLevi}\label{prop:LiesmallLevi}
There exists some $\consta\label{k:Levi-exp} > 0$ with the following property.
Let $\fg\subset\fsl_N(\Q)$ be a Lie subalgebra and let $\fr=\rad(\fg)$ be its radical.
Further, let $\mathfrak l\subset\gfrak$ be a reductive subalgebra with $\mathfrak l\cap \fr=\{0\}$ (it may be that $\mathfrak{l} = \{0\}$). 
Assume that $\height(\gfrak)\leq T$ and $\height(\mathfrak l)\leq T$. 
There exists a Levi decomposition $\fg = \fh\oplus \fr$ with $\mathfrak l\subset\fh$, so that 
\[
\height(\fh) \ll T^{\ref{k:Levi-exp}}\quad\text{ and }\quad\height(\fr) \ll T^{\ref{k:Levi-exp}},
\]
where the implied constants depend only on $N$.
\end{propo}

Roughly speaking, the proof of the proposition is based of the following phenomenon: a {\em consistent} system of linear equations with integral coefficients which are bounded by $T$ has a solution of norm $\ll T^\star$. 

Let us also note that if $\rad(\G)=\rad_u(\G)$ and $\gfrak=\Lie(\G)$, 
the condition $\mathfrak l\cap \mathfrak r=\{0\}$ holds true for any reductive subalgebra. 

\subsection{Systems of integral linear equations}\label{sec:small-sol}
For the convenience of the reader, in this section we record some lemmas which provide estimates on the size of solutions of systems of linear equations with integral coefficients.

We note that the following lemmas aim for {\em good} polynomial bounds. 
If one is content with a rough polynomial bound,  
one could easily prove
\[
|x_{j}^l| \leq \sqrt{(N-1)!} \cdot (\max_{ij} |a_{ij}|)^{N-1}
\]
in the first lemma and the bound $\| v_i \| \leq N \max_{j} \|u_j\|$ in the third lemma --- these rough bounds suffice for our applications as well.

\begin{lem}[Siegel's lemma] \label{Siegel}
Let $A = (a_{ij})$ be a $M \times N$-matrix ($N>M$) of full rank, with integer coefficients $a_{ij}$, and
\[
\begin{cases} \sum_{j=1}^N a_{ij} x_j = 0 \qquad i= 1, \dots, M \end{cases}
\]
the associated linear system. There exists a basis $\{(x_{1}^{l}, \dots, x_{N}^l) \mid l=1, \dots, N-M \}$ of the space of solutions of the system satisfying $x_{j}^l \in \Z$ and $|x_{j}^l| \leq \sqrt{|\det A A^{T}|}$ for $l=1, \dots N-M$. 
\end{lem}

\begin{proof}
See \cite[Thm.~2]{BombieriVaaler83}. 
\end{proof}

\begin{lem}[Siegel's lemma for inhomogeneous equations] \label{inSiegel}
Let 
\[
\begin{cases} \sum_{j=1}^N a_{ij} x_j = b_i \qquad i= 1, \dots, M \end{cases}
\]
be a consistent system of $M$ linear equations in $N>M$ variables, with integer coefficients $a_{ij}$. Then the system has a solution $(\frac{y_1}{d}, \dots, \frac{y_N}{d})$ with $y_i, d \in \Z$ and 
\[
\max_i \{|y_i|, |d|\} \ll(\max_{ij} |a_{ij}|)^\star.
\] 
\end{lem}

\begin{proof}
The lemma is deduced from \cite[Thm.~2 and 3]{OLearyVaaler93}. 
First, by assumption, the system has a solution $(\frac{y_1}{z_1}, \dots, \frac{y_N}{z_N})$ in $\Q^N$. 
Set $P = \{p\in\places_f : p | z_i \textrm{ for some $1\leq i\leq N$}\}\cup \{\infty\}$. 
Then \cite[Thm.~2 and 3]{OLearyVaaler93} apply to our system and the set $P$ of places, and yield a solution of the system with bounded height. 

The bound on the height is independent of $P$, and in our setting, it readily translates to a polynomial bound on $\max_i \{|y_i|, |z_i|\}$. 
\end{proof}

\begin{lem}[extracting small $\Z$-bases] \label{smallZbasis}
Let $V$ be a vector space over $\Q$ endowed with a norm $\| \cdot \|$ and let $V_\Z$ be a free $\Z$-submodule of $V$ which spans $V$ over $\Q$. Given a basis $\{u_1, \dots, u_N\}$ of $V$ over $\Q$ lying in $V_\Z$, there exists a subset $\{v_1, \dots, v_N\}$ of $V_\Z$ with the property that $\{v_1, \dots, v_i\}$ is a $\Z$-basis of $(\Q u_1 + \dots + \Q u_i) \cap V_\Z$ and $\|v_i\| \leq \sum_{j=1}^i \|u_j\|$ for $i = 1, \dots, N$. 
\end{lem}

\begin{proof}
Let $\{v_1, \dots, v_N\}$ be a $\Z$-basis of $V_\Z$ and $A = (a_{ij})$ be the integer matrix such that $(u_1, \dots, u_N) = (v_1, \dots, v_N) A$. Up to a change of the basis $\{v_1, \dots, v_N\}$, we may assume that $A$ is in Hermite normal form, i.e., $A$ is upper triangular, all its entries are non-negative, and in a given column, the entry on the diagonal is strictly bigger than the other ones. 
We then have
\[
\begin{cases}
u_1 = a_{11} v_1 \\
u_2 = a_{12}v_1 + a_{22} v_2 \\
\qquad \vdots \\
u_m = a_{1m} v_1 + \dots + a_{NN} v_N.
\end{cases} 
\]
If $w \in (\Q u_1 + \dots + \Q u_i) \cap V_\Z$, we may write $w = \sum_{j=1}^N \lambda_j v_j$ with $\lambda_j \in \Z$. Now
\[
w - \sum_{j=1}^i \lambda_j v_j = \sum_{j=i+1}^N \lambda_j v_j \in (\Q v_1 + \dots + \Q v_i) \cap (\Q v_{i+1} + \dots + \Q v_N) = \{0\},
\]
and it follows that $w = \sum_{j=1}^i \lambda_j v_j$, i.e.~$\{v_1, \dots, v_i\}$ is a $\Z$-basis of $(\Q u_1 + \dots + \Q u_i) \cap V_\Z$. 

Lastly, regrouping terms and taking norms in the system above yields
\[
\begin{cases}
\|v_1\| = \frac{\|u_1\|}{a_{11}} \leq \| u_1 \| \\
\|v_2\| = \frac{\| u_2 - a_{12}v_1\|}{a_{22}} < \|u_2\| + \|v_1\| \\
\qquad \vdots \\
\|v_m\| = \frac{\| u_m - a_{1N}v_1 - \dots - a_{(N-1) N} v_{N-1}\|}{a_{NN}} < \|u_2\| + \|v_1\| + \dots \|v_{N-1}\|.
\end{cases}
\]
The lemma follows by combining all the inequalities. 
\end{proof}

\subsection{Proof of Proposition~\ref{prop:LiesmallLevi}}\label{sec:proof-prop-levi}
We need to find a Levi decomposition $\fg = \fh\oplus \fr$, where $\fr$ is the radical of $\gfrak$, and $\Z$-bases $\{w_1, \dots, w_n\}$ of $\fh \cap \fsl_N(\Z)$, and $\{v_1, \dots, v_m\}$ of $\fr \cap \fsl_N(\Z)$ which satisfy that 
\[
\text{$\|v_i \| \leq T^\star$ and $\|w_j\| \leq T^\star$ for all $i, j$. }
\]

If $\mathfrak l \neq 0$, let $\{\tilde u_1,\ldots, \tilde u_l\}$ be a $\Z$-basis for $\mathfrak l\cap\fsl_N(\Z)$ with $\|\tilde u_i\|\ll T^\star$.
Extend this to a $\Q$-basis $\tilde{\mathcal B} = \{\tilde u_1, \dots, \tilde u_M \}\subset \fg \cap \fsl_N(\Z)$ for $\fg$ with $\|\tilde u_i\| \ll T^\star$ for all $i$.
By the {\em extracting small $\Z$-bases} lemma in \S\ref{smallZbasis}, there exists a $\Z$-basis 
$\hat{\mathcal B}=\{\hat u_1\dots,\hat u_M\}$ for $\fg\cap\fsl_N(\Z)$ so that
$\|\hat u_i\| \leq T^\star$ for all $i$ and $\{\hat u_1,\ldots,\hat u_i\}$ is a $\Z$-basis for $(\Q \tilde u_1 + \dots + \Q \tilde u_i) \cap \fsl_N(\Z)$.
In particular, $\{\hat u_1,\ldots,\hat u_l\}$ is a $\Z$-basis for $\mathfrak l\cap\fsl_N(\Z)$ if $\mathfrak l\neq0$.

Note that the structure constants $\{\alpha_{ij}^k \}$ of $\fg$ in the basis $\{u_i\}$ are bounded: 
\be\label{eq:str-cont}
\max_k | \alpha_{ij}^k | = \| [\hat u_i,\hat u_j]\|_{\hat{\mathcal B}} \ll (\max_j{\|\hat u_j\|}) \cdot \| [\hat u_i,\hat u_j]\| \ll T^\star.
\ee
As $\hat{\mathcal B}$ is a $\Z$-basis for $\fg \cap \fsl_N(\Z)$, the $\{\alpha_{ij}^k \}$ are integers. 

\medskip

{\em Step 1.}\ Bounding $\height(\fr)$.

Let $\mathsf k$ denote the killing form of $\gfrak$. Recall that the radical $\fr=\rad(\gfrak)$ is the orthogonal complement of the derived algebra 
$[\fg,\fg]$ for $\mathsf k$. Thus $\fr$ is given in the basis $\hat{\mathcal B}$ by the solutions $(y_i)$ of the system 
\[
\mathsf k\biggl(\sum_{i=1}^M y_i \hat u_i, [\hat u_j,\hat u_k]\biggr) = 0,\quad j,k = 1, \dots, M.
\]
The coefficients of this system are $\ll T^\star$.
%
Thus, after removing redundant equations from the system, we may apply {\em Siegel's lemma} combined with {\em extracting small $\Z$-bases lemma} from \S\ref{sec:small-sol} and obtain the following.
There exists a $\Z$-basis $\{v_1, \dots, v_m\}$ of $\fr \cap \fsl_N(\Z)$, so that $\|v_i\|_{\hat{\mathcal B}} \ll T^\star$. In consequence, 
we get that
\be\label{eq:ht-rad}
\|v_i\| \ll (\max_j{\|\hat u_j\|}) \cdot \|v_i\|_{\hat{\mathcal B}} \ll T^\star.
\ee

{\em Step 2.}\ A basis for $\gfrak$ adapted to $\mathfrak l$, $\fr$, and $[\fr,\fr]$.

Let $\{v_1, \dots, v_m\}$ be a $\Z$-basis of $\fr \cap \fsl_N(\Z)$ as constructed above. 
We first gather a basis of $[\fr,\fr]$ among $\{[v_i, v_j] \mid i,j = 1, \dots, m\}$, then extend this to a $\Q$-basis, $\mathcal C$,
of $\fr$ by adding an appropriate subset of $\{v_1, \dots, v_m \}$ to it. 
Finally, we extend $\mathcal C$ to a $\Q$-basis, $\mathcal B'$, of $\fg$ by adding an appropriate subset of 
$\{\hat u_1,\dots, \hat u_M\}$ to $\mathcal C$. 
Note that if $\mathfrak l\neq0$, we may obtain $\{\hat u_1,\dots,\hat u_l\}\subset \mathcal B'$ because $\mathfrak l\cap\fr=\{0\}$. 

Applying the {\em extracting small $\Z$-bases lemma} from \S\ref{sec:small-sol} yields a $\Z$-basis 
$\mathcal B=\{u_1, \dots, u_M\}$, of $\fg \cap \fsl_N(\Z)$ so that
\begin{enumerate}
\item $\|u_i\|\ll T^\star$.
\item $\{u_1,\ldots,u_{m'}\}$ is a $\Z$-basis for $[\fr,\fr] \cap \fsl_N(\Z)$.
\item $\{u_1,\ldots, u_m\}$ is a $\Z$-basis for $\fr \cap \fsl_N(\Z)$.
\item $\{u_1,\ldots,u_{m+l}\}$ is a $\Z$-basis for $(\mathfrak l\oplus \rfrak) \cap \fsl_N(\Z)$.
\end{enumerate} 
In particular, $\{u_{m'+1}, \dots, u_{m} \}$ projects to a basis of $\rfrak/ [\fr,\fr]$. 
Let us write $\mathcal D:=\{u_{m'+1} + [\fr,\fr], \dots, u_{m} + [\fr,\fr] \}$.

Also note that for $1\leq i\leq l$ and $1\leq j\leq m+l$, 
there are $c_{ij}\in\Z$ with $|c_{ij}|\ll T^\star$ so that for each $1\leq i\leq k$ we have 
\be\label{eq:uhat-u}
\hat u_i=\sum_{j=1}^{m+l} c_{i,j}u_j.
\ee  

\medskip

{\em Step 3.}\ Finding a Levi subalgebra $\fh$ with small height. 

We argue by induction on $\ell_{\rm d}(\fr)$, the derived length of the radical $\fr$. When $\ell_{\rm d}(\fr)=0$, $\fg$ is semisimple, and it suffices to set $\fh = \fg$. 

Therefore, let us assume that $\ell_{\rm d}(\fr)\geq 1$. Define
\[
E = \{ f \in {\rm End}(\fg,  \fr /\mathfrak  [\fr,\fr]) : \text{$f$ satisfies (a), (b), and (c)}\},
\]
where ${\rm End}(\fg, \fr / [\fr,\fr])$ 
denotes the set of $\Q$-linear maps from $\fg$ to $\fr /[\fr,\fr]$,  
and
\begin{enumerate}
\item[(a)] $\mathfrak l\subset\ker f$,
\item[(b)] $f$ restricts to the canonical projection $\fr \to \fr / [\fr,\fr]$,
\item[(c)] $f([u,v]) = [u+[\fr,\fr],f(v)] + [f(u),v+[\fr,\fr]]$ for all $u,v \in \fg$.
\end{enumerate}

If $\fh$ is a Levi subalgebra of $\fg$ which contains $\mathfrak l$, then the canonical projection 
$\fg = \fr \oplus \fh \to \fr/ [\fr,\fr]$ (whose kernel is precisely $[\fr,\fr] \oplus \fh$) belongs to $E$. 
Now, since $\mathfrak l$ is reductive, there exists a Levi subalgebra $\mathfrak h$
so that $\mathfrak l\subset\fh$, see~\cite{Mostow}. Therefore, $E\neq \emptyset$.

\begin{claim} 
If $f \in E$, then $\ker f$ is a Lie subalgebra of $\fg$ whose radical is $[\fr,\fr]$. 
\end{claim}

\begin{proof}[Proof of the claim]
First note that in view of (c) above, $\ker f$ is a subalgebra.
Also, it is clear from (b) that $[\fr,\fr] \subset \rad (\ker f)$. 

To see the converse, note that $\fr + \ker f = \fg$, hence $\fr + \rad (\ker f)$ is an ideal of $\fg$.
Moreover, $\fr + \rad (\ker f)$ is solvable.  
Therefore, $\rad(\ker f) \subset \fr \cap \ker f = [ \fr,\fr]$, where the last equality follows from~(b). 
\end{proof}

In view of the claim, if $\fh$ is a Levi subalgebra of $\ker f$ with $\mathfrak l\subset\fh$, then 
\be\label{eq:levi-kerf}
\fg = \fr + \ker f = \fr + ([\fr,\fr] \oplus \fh) = \fr \oplus \fh.
\ee
That is: $\fh$ is a Levi subalgebra of $\fg$ and $\mathfrak l\subset\mathfrak h$. 

The strategy now is to find some $f\in E$ with $\height(\ker(f))\ll T^\star$.
Then the above observation and inductive hypothesis will yield the desired Levi subalgebra. 

We now turn to the details. 
First note that in view of (a), (b) and (c), we have that
$E$ is the set of solutions $f \in {\rm End}(\fg,\fr/ [\fr,\fr])$ of the {\em inhomogeneous} system
\[
\begin{cases}
f(u_i)=0&\quad i=1,\ldots,m'\\
f(u_i) = u_i +[\fr,\fr] &\quad i=m'+1, \dots, m \\
f(\hat u_i)=0&\quad i=1,\dots,l\\
f([u_i, u_j]) = [u_i +[\fr,\fr] ,f(u_j)] + [f(u_i),u_j +[\fr,\fr] ] &\quad i,j = 1, \dots, M.
\end{cases}
\]
In view of~\eqref{eq:str-cont} and~\eqref{eq:uhat-u} we have the following. 
When $f$ is written in the basis of ${\rm End}(\fg,\fr/ [\fr,\fr])$ associated to $\mathcal B$ and $\mathcal D$,
the above system becomes a linear system whose coefficients are integers bounded in absolute value by $\ll T^\star$.  

Since $E$ is not empty, after perhaps removing redundant equations, 
we may apply {\em Siegel's lemma for inhomogeneous equations} in \S\ref{inSiegel} and get the following. 
There is a solution $f$ whose matrix in the bases $\mathcal B$ and $\mathcal D$ 
has rational entries, with numerator and common denominator $c\ll T^\star$. 
Put $f' = c f$, so that the matrix of $f'$ in the bases $\mathcal B$ and $\mathcal D$ has integer coefficients of size $\ll T^\star$. 

At last, another application of {\em Siegel's lemma} and {\em extracting small $\Z$-bases lemma} in \S\ref{inSiegel} to $f'$
yields that $\ker f \cap \fsl_N(\Z)$ has a $\Z$-basis $\{w_1, \dots, w_{n'} \}$ satisfying 
\[
\| w_i \| \ll T^\star\quad 1\leq i\leq n'
\] 

Recall from the claim that $\mathfrak l\subset\ker f$, $\rad(\ker(f))=[\fr,\fr]$, and $\ell_{\rm d}([\fr,\fr])<\ell_{\rm d}(\fr)$. 
Hence by the inductive hypothesis, $\ker f$ has a Levi subalgebra, $\fh$, with $\height(\fh)\ll T^\star$.

In view of~\eqref{eq:levi-kerf}, this finishes the proof of {\em Step 3} and the proposition.
\qed

\section{Consequences of effective Levi decomposition}\label{sec:eff-levi}
Recall from \S\ref{sec:data} that we fixed the following.
\begin{enumerate}
\item A connected, simply connected, algebraic $\bbq$-group $\G$
whose solvable radical is unipotent, i.e.,~$\rad(\G)=\rad_u(\G)=:\bR$.
\item An algebraic homomorphism $\iota: \G \rightarrow \SL_N$ 
defined over $\bbq$ with a finite central kernel.
\end{enumerate}

Also recall that $\mu_Y$ (or simply $\mu$) denotes the $G=\iota(\G(\adele))$-invariant probability measure on $Y=\iota(\G(\adele)/\G(F))$. 
Let $m_G$ (or simply $m$) be a Haar measure on $G$ which projects to $\mu$ under the orbit map. 

In this section, we will use the results from \S\ref{sec:small-levi} to find a {\em good} Levi decomposition for $\iota(\G)$.
Then we will relate the notion of height of $Y$ (see \S \ref{sec:volume}, \S\ref{sec:height}) to the heights of orbits similarly defined using the radical and our fixed Levi subgroup.  

\subsection{Finding a good Levi subgroup}\label{sec:Levi-iota-G} 
Let $\fg' \subset \fsl_N(\Q)$ (resp.~$\fr'$) denote the Lie algebra of $\iota(\bG)$ (resp.~of $\iota(\bR)$). 
Set $T := \height(\fg')$. 

Let $\fh'$ be a Levi subalgebra of $\fg'$ given by proposition \ref{LiesmallLevi} applied to $\fg'$, so that $\height(\fh')\ll T^\star$. 

Let $\bH'$ be the subgroup of $\iota(\G)$ with $\Lie(\bH')=\fh'$. Then $\bH'$ is a Levi subgroup of $\iota(\G)$,
and we have $\iota(\G)=\bH'\iota(\bR)$. 

We now discuss similar decompositions over $\Q$ and also $\A$. 
First note that, since $\bR = \rad_u(\bG)$, $\bH'$ is semisimple (not just reductive), and we have $\bH'(\A) \cap \iota(\bR)(\A) = \{1\}$. 

Set $\bH = \iota^{-1}(\bH')$. Since $\iota$ has finite central kernel, $\bH$ is a semisimple $\Q$-subgroup of $\bG$ isogenous to $\bH'$; thus $\bH$ is a Levi $\Q$-subgroup of $\bG$. Moreover, $\bG(\Q) = \bH(\Q) \bR(\Q)$. Indeed, in the exact sequence 
\[
1 \rightarrow\bR(\Q) \rightarrow \bG(\Q) \rightarrow\bH(\Q) \rightarrow\rH^1(\Q, \bR)
\]
associated to the quotient $\bG / \bR \cong \bH$, the term $\rH^1(\Q, \bR)$ vanishes because $\bR$ is unipotent. 
Hence $\bG(\Q) \to \bH(\Q)$ is onto. 

The same argument applied to the group $\iota(\G)$ shows that $\iota(\bG)(\Q) = \iota(\bH)(\Q) \iota(\bR)(\Q)$.

The above also implies that
\[
\bG(\A) = \bH(\A) \bR(\A).
\] 
Indeed, since $\bG(\Q) = \bH(\Q) \bR(\Q)$, 
the embedding $\bH \to \bG$ is a section defined over $\Q$ of the quotient map $\bG \to \bH$. 
Hence $\bG(\A) \to \bH(\A)$ is surjective, see \cite[\S1.2]{Weil82}, and 
we get
$\bG(\A) = \bH(\A)\bR(\A)$ as was claimed. 

Applying $\iota$, this yields 
$
\iota(\bG(\A)) = \iota(\bH(\A)) \iota(\bR(\A)).
$

\subsection{Product structure of $Y,$ $\mu_Y,$ and $m_G$} \label{sec:prod-str}
Let $\hat{\rm pr}_{\bH}:\G\to\bH$ be the map which is induced from the natural projection $\G\to\G/\bR$.
More explicitly, given $g\in \G$, we have the unique decomposition 
\[
g=g_\bH g_\bR\quad\text{where } g_\bH\in\bH\text{ and }g_\bR\in\bR;
\] 
then $\hat{\rm pr}_{\bH}(g)=g_{\bH}$. 

Let ${\rm pr}_H:G\to H := \iota(\bH(\A))$ be the induced map, given by ${\rm pr}_H(g)=g_H$, where $g=g_Hg_R\in\iota(\bH(\A))\iota(\bR(\A))$.  

Put $Y_H :=\iota(\bH(\A)/\bH(\Q))$. The map ${\rm pr}_H$ 
induces a map $Y\to Y_H$ given by $\iota(g)\SL_N(\Q) \mapsto \iota(g_\bH) \SL_N(\Q)$ for $g \in \bG(\A)$.
To see this, suppose $\iota(g_1^{-1}g_2) \in \SL_N(\Q)$ for some $g_1, g_2 \in \bG(\A)$. 
Then $\iota(g_1^{-1}g_2) \in \iota(\bG)(\Q)=\iota(\bH)(\Q)\iota(\bR)(\Q)$, hence 
\[
\iota({({g_1})_\bH}^{-1} (g_2)_\bH) = \iota((g_1^{-1}g_2)_\bH) \in \iota(\bH)(\Q) \subset \SL_N(\Q).
\]
We continue to denote the map so induced from $Y$ to $Y_H$ by ${\rm pr}_H$.

Put $R := \iota(\bR(\A))$ and $Y_R :=\iota(\bR(\A)/\bR(\Q))$;
we have a fibration
\[
\begin{tikzcd}
Y_R \arrow[hookrightarrow]{r} & Y \arrow[d,"{\rm pr}_H"] \\
& Y_H
\end{tikzcd}
\]
The fiber over $\iota(h)\SL_N(\Q) \in Y_H$ is ${\rm pr}_H^{-1}(\iota(h)) = \iota(h) \iota(\bR(\A)) \SL_N(\Q)$, 
the translate of $Y_R$ by $\iota(h)$. 

Let $\mu_R$ (resp.~$\mu_H$) be a $R$-invariant (resp.~$H$-invariant) probability measure on $Y_R$ (resp.~$Y_H$). Let $\hat\mu$ be the measure on $Y$ defined by
\[
\int_{Y} f \operatorname{d}\!\hat\mu = \int_{Y_H} \left( \int_{Y_R} f(hr\SL_N(\Q)) \operatorname{d}\!\mu_R(r) \right) \operatorname{d}\!\mu_H(h).
\]

Since $\bH$ is semisimple, the modulus of the action of $H$ on $\mu_R$ is trivial. Thus $\hat\mu$ is a $G$-invariant probability measure on $Y_G$; that is: $\hat\mu = \mu_Y$. 

Let $m_H$ and $m_R$ be Haar measures on $\iota(\bH(\A))$ and $\iota(\bR(\A))$ which project to $\mu_H$ and $\mu_R$, respectively.
The measure $\hat m$ on $G$ given by the product of $m_H$ and $m_R$ is a Haar measure. 
Moreover, $\hat m$ projects to the invariant probability measure $\hat \mu=\mu_Y$ on $Y$ via the orbit map.
Therefore, $m_G=\hat m$ is the product of $m_H$ and $m_R$.

\begin{lemma}\label{lem:ht-vol-unip}
There exists some $\consta\label{k:ht-vol}$ so that
\[
\vol(Y_R)^{\ref{k:ht-vol}}\ll\height(\bR)\ll \vol(Y_R). 
\]
\end{lemma}

\begin{proof}
Recall that $\iota$ is an isomorphism on $\bR$.
For every prime $p$, put 
\[
C_p=\iota^{-1}(\bR'(\Q_p)\cap\SL_N(\Z_p));
\]
$C_p$ is a compact open subgroup of $\bR(\Q_p)$. 
By the strong approximation theorem for unipotent groups, we have
\[
\bR(\A)=\bigl(\bR(\bbr)\times\textstyle\prod_{p}C_p\bigr)\bR(\Q).
\] 
In other words, for every $g\in \bR(\A)$ there exists some $\gamma_0\in\bR(\Q)$ so that 
\[
g\gamma_0=(\hat g_\infty,(\hat g_p))\in\bR(\bbr)\times\textstyle\prod_{p}C_p.
\]

Recall now that $\log(\iota(\bR(\bbr))\cap\SL_N(\Z))\subset (\mathfrak r'\cap\frac{1}{D}\fsl_N(\Z))$
for some integer $D$ depending only on $N$.
Let $\{v_1,\ldots, v_n\}$ be a $\Z$-basis for $\mathfrak r'\cap\frac{1}{D}\fsl_N(\Z)$. 
For every $\delta>0$, put 
\[
F_\delta :=\{h_\infty \in \iota(\bR(\R)):h_{\infty}=\exp(\sum c_iv_i), |c_i|<\delta\}\} \times \prod_{p} \iota(C_p).
\]
Note that, 
\[
F_\delta\cap\SL_N(\Q)\subset\SL_N(\Z);
\] 
therefore, in view of the choice of $D$, for small enough $\delta\ll1$
we have $F_\delta\cap\SL_N(\Q)=\{e\}$.

Also, note that $F_\delta=F_\delta^{-1}$; and if $\delta\ll1$ is small enough, 
$hh'\in F_{\star\delta}$ for any $h,h'\in F_\delta$.
Altogether, we get that $F_\delta$ injects into $Y_R$ for all small enough $\delta\ll1$. 

Recall that $m_R$ is a Haar measure on $\iota(\bR(\A))$ normalized so that $\mu_R(Y_R)=1$;
also recall that $\Omega=\exp(B_{\fsl_N(\R)}(\eta_0))\times \prod_p\SL_N(\Z_p)$. Therefore, 
\[
m_R(\iota(\bR(\A))\cap\Omega)\ll \|v_1\wedge\dots\wedge v_n\|^{-1}.
\]
Since $\|v_1\wedge\dots\wedge v_n\|\asymp\height(\bR)$, we get from the above that 
\[
\height(\bR)\ll \vol(Y_R)=m_R(\iota(\bR(\A))\cap\Omega)^{-1}.
\]

To see the other inequality, let $g\in \bR(\A)$. Let $\gamma_0\in\bR(\Q)$ be so that 
\[
g\gamma_0=(\hat g_\infty,(\hat g_p))\in\bR(\bbr)\times\textstyle\prod_{p}C_p.
\]
There exists some $\hat\gamma_1\in\iota^{-1}(\exp(\fr' \cap N!\fsl_N(\Z)))$ 
so that 
\[
|\iota(\hat g_\infty\hat\gamma_1)|\ll \height(\bR)^{\kappa}
\] 
for some $\kappa$ independent of $g$. Note that $\iota(\hat\gamma_1) \in \iota(\bR(\R)) \cap \SL_N(\Z)$, hence $\hat\gamma_1 \in \bR(\Q)$. 

Let $\gamma_1$ be the diagonal embedding of $\hat\gamma_1$ in $\bR(\A)$. Then 
since $\iota(\hat\gamma_1)\in \SL_N(\Z)$, we get that 
\[
g\gamma_0\gamma_1=(\hat g_\infty,(\hat g_p))\gamma_1=(\hat g_\infty\hat\gamma_1, (\tilde g_p))\in\bR(\bbr)\times\textstyle\prod_{p}C_p. 
\] 
Since we can cover $\{g\in \bR(\bbr): |\iota(g)|\ll \height(\bR)^{\kappa}\}$ with $\ll \height(\bR)^\star$ translates of 
$\iota(\bR(\bbr))\cap \exp(B_{\fsl_N(\R)}(\eta_0))$, we get that 
\[
m_R(\iota(\bR(\A))\cap\Omega)\gg \height(\bR)^{-\star}.
\]
Therefore, $\vol(Y_R)^\star\ll \height(\bR)$; the proof is complete.
\end{proof}

\begin{lemma}\label{lem:height-g-g-H}
There exist $\consta\label{k:ht-g-gH-1}$ so that the following holds.
For any $g\in G$ we have
\[
\height(\G)^{-\ref{k:ht-g-gH-1}}\height(g)^{\ref{k:ht-g-gH-1}}\ll\height({\rm pr}_H(g))\ll\height(\G)^{\ref{k:ht-g-gH-1}}\height(g)^{\ref{k:ht-g-gH-1}}.
\]
\end{lemma}

\begin{proof}
Recall our notation $\G'=\iota(\G)$ 
and the Levi subgroup $\bH'$ of $\G'$ from \S\ref{sec:Levi-iota-G}. Put $\bR':=\iota(\bR)=\rad_u(\G')$. 
If $\bR'=\{1\}$, then $\bR=\{1\}$ and there is nothing to prove. 
Therefore, let us assume that $\bR'$ is a nontrivial unipotent $\Q$-subgroup of $\SL_N$.

Let ${\bf P}\subset\SL_N$ be the parabolic subgroup associated to $\bR'$ as in~\cite{BT}.
That is: ${\bf U}_0=\bR'$ and 
${\bf U}_i$ is defined inductively by $\rad_u(\rN_{\SL_N}({\bf U}_{i-1}))$. 
Then ${\bf U}_i\subset \rN_{\SL_N}({\bf U}_{i-1})$ and ${\bf U}_{i-1}\subset{\bf U}_i$.
This process terminates after $d\leq N^2$ steps
and gives rise to a parabolic subgroup, $\bf P$, with the following properties. 
\begin{enumerate}
\item $\height({\bf P})\ll \height(\bR)^\star$.
\item $\bR'\subset \rad_u({\bf P})=:{\bf W}$.
\item $\rN_{\SL_N}(\bR')\subset{\bf P}$.
\end{enumerate}
In view of~(1) and Proposition~\ref{LiesmallLevi}, we have $\height({\bf W})\ll \height({\bf P})^\star\ll\height({\bR})^\star$.
Moreover, by (3) we have $\G'\subset{\bf P}$.

Let $\mathcal F_{\bf P}$ denote the flag defined by ${\bf W}$ as follows.
Let $\mathcal V_0=\Q^N$, and for any $m>0$, let 
\[
\mathcal V_m=\text{$\Q$-span}\{w_1\ldots w_m v: v\in\Q^N, w_i\in\Lie({\bf W})\}.
\]
Then $\{\mathcal V_m\}$ forms a descending chain of subspaces of $\Q^N$; 
let $M\leq N$ be so that $\mathcal V_{M}\neq0$ but $\mathcal V_{M+1}=0$.
Further, note that $\height(\mathcal V_m)\ll\height(\bR)^\star$ for each $0\leq m\leq M$.

There exists some $\delta=(\tfrac{a_{ij}}{b_{ij}})\in\SL_N(\Q)$ with 
$|a_{ij}|,|b_{ij}|\ll \height(\bR)^\star$ so that $\delta\mathcal F_{\bf P}=\mathcal F_0$ where $\mathcal F_0$
is a standard flag, i.e., $\mathcal F_0$ is a flag corresponding to a block upper triangular parabolic subgroup ${\bf P}_0$.   
One could construct one such $\delta$ as follows: for each $i\geq0$ let $\mathcal V'_{i}$ be a complement of
$\mathcal V_{M+1-i}$ in $\mathcal V_{M-i}$ (in particular, $\mathcal V'_0=\mathcal V_M$), chosen so that
$\height(\mathcal V'_i)\ll\height(\bR)^\star$ for all $i$.

Let us put ${\bf Q}_0=[{\bf P}_0,{\bf P}_0]\rad_u({\bf P}_0)$. 
The group $\rad_u({\bf P}_0)$ is unipotent upper triangular and since $\delta{\bf W}\delta^{-1}\subset\rad_u({\bf P}_0)$, 
we have $\delta\bR'\delta^{-1}\subset \rad_u({\bf P}_0)$. 
Further, $\delta\bH'\delta^{-1}\subset {\bf Q}_0$ since $\bH'$ is perfect and normalizes $\bR'$.

Let $g\in G\subset \G'(\A)$; write $g=g_Hg_R$ where $g_H\in\iota(\bH(\A))$ and $g_R\in\iota(\bR(\A))$ ---
recall that ${\rm pr}_H(g)=g_H$.
We will use the reduction theory of $\SL_d$ to compute an Iwasawa decomposition for 
representatives of $\hat g_H:=\delta g_H\delta^{-1}$ and $\hat g:=\delta g\delta^{-1}$ in a Siegel fundamental domain. 

Decompose $\hat g_H$ as a product of a block-diagonal matrix in ${\bf Q}_0$ and an element in 
$\rad_u({\bf P}_0)$. Then using the reduction theory of $\SL_d$ for each block matrix and the fact that
$\rad_u({\bf P}_0)$ is normal subgroup of ${\bf Q}_0$, we have the following. 
There exists some $\gamma_0\in{\bf Q}_0(\Q)$ so that 
\be\label{eq:gH-iwasawa}
\hat g_H\gamma_0=(k a u,(g'_{H,p}))\in \SL_N(\R)\times\bigl(\textstyle\prod_p\SL_N(\Z_p)\bigr)
\ee
where $k\in{\SO}_N(\R)$, $a={\rm diag}(a_i)$ is diagonal with
$a_ia_{i+1}^{-1}\leq 2/\sqrt3$, and $u=(u_{ij})$ is unipotent upper triangular with $|u_{ij}|\leq 1/2$ 
(see also the proof of Lemma~\ref{lem:ht-injr}).

Let us write $\hat g_R=\delta g_R\delta^{-1} \in \rad_u({\bf P}_0)(\A)$. 
Let $\gamma_1 \in \SL_N(\Q)$ be unipotent upper triangular, such that 
\[
(u,(e)) \gamma_0^{-1}\hat g_R\gamma_0\gamma_1=(u', (u'_p))
\] 
with $u'=(u'_{ij})$ and $|u'_{ij}|\leq 1/2$, and $u'_p\in \SL_N(\Z_p)$. 
This, in view of~\eqref{eq:gH-iwasawa}, gives
\begin{align}
\hat g \gamma=\notag \hat g_H\hat g_R \gamma &= \hat g_H\gamma_0(\gamma_0^{-1}\hat g_R\gamma_0\gamma_1) = (ka, (g'_{H,p})) (u,(e)) \gamma_0^{-1}\hat g_R\gamma_0\gamma_1\\
\label{eq:gH-iwasawa-1} &=( k a u',(g'_{H,p} u'_p))\in \SL_N(\R)\times\bigl(\textstyle\prod_p\SL_N(\Z_p)\bigr)
\end{align}
where $\gamma=\gamma_0\gamma_1$.

As was discussed in the proof of Lemma~\ref{lem:ht-injr}, the decompositions in~\eqref{eq:gH-iwasawa} and~\eqref{eq:gH-iwasawa-1} 
imply that
\[
|a|^\star\ll \height(\hat g_H)\ll |a|^\star\quad\text{and}\quad|a|^\star\ll \height(\hat g)\ll |a|^\star.
\]
Recall now that $g=\delta^{-1}\hat g\delta$ and $g_H=\delta^{-1}\hat g_H\delta$ 
where $\delta=(\tfrac{a_{ij}}{b_{ij}})\in\SL_N(\Q)$ with $|a_{ij}|,|b_{ij}|\ll \height(\bR)^\star\ll\height(\G)^\star$.
The claim thus follows. 
\end{proof}

\begin{propo}\label{prop:vol-comp}
There exist $\consta\label{k:col-comp-exp}$ and $\consta\label{k:col-comp-exp-up}$ with the following property. 
\begin{enumerate}
\item 
$
\bigl(\vol(Y_H) \vol(Y_R)\bigr)^{\ref{k:col-comp-exp}} \ll \disc(Y) \ll \bigl(\vol(Y_H) \vol(Y_R)\bigr)^{\ref{k:col-comp-exp-up}}. 
$
\medskip
\item
$
\bigl(\disc(Y_H) \disc(Y_R)\bigr)^{\ref{k:col-comp-exp}} \ll \disc(Y) \ll \bigl(\disc(Y_H) \disc(Y_R)\bigr)^{\ref{k:col-comp-exp-up}}. 
$
\end{enumerate}
\end{propo}

\begin{proof}
Recall definitions of $\vol(\cdot)$ and $\height(\cdot)$  of an algebraic data from~\eqref{eq:volume} and~\eqref{eq:def-discY}, respectively. 

We first show that part~(2) follows from part~(1). 
Indeed by Lemma~\ref{lem:ht-vol-unip}, we have $\height(\bR)\ll \vol(Y_R)$; hence, $\disc(Y_R)\asymp \vol(Y_R)$.
Moreover, by~\cite[App.~B]{EMMV} we have $\height(\bH)^\star\ll \vol(Y_H)$; hence, $\vol(Y_H)^\star\ll\height(Y_H)\ll \vol(Y_H)^\star$.

\medskip

We now turn to the proof of part~(1) in the proposition. 

\medskip

{\em The upper bound.}\ 
Because multiplication is Lipschitz (alternatively, by the Baker-Campbell-Hausdorff formula), perhaps after changing $\eta_0$, we have $\Omega_\eta \cdot \Omega_\eta \subset \Xi_{c \eta} \times \prod_{v \in \Sigma_f} \SL_N(\fo_v) = \Omega_{c \eta}$ for some $c$ depending only on $N$, hence 
\[
\left( \iota(\bH(\A)) \cap \Omega_\eta \right) \cdot \left( \iota(\bR(\A)) \cap \Omega_\eta \right) \subset \iota(\bG(\A)) \cap \Omega_{c \eta}.
\]
In view of our discussion in~\S\ref{sec:prod-str}, the measure of the left hand side is
\[
m_G \bigl( \left( \iota(\bH(\A)) \cap \Omega_\eta \right) \cdot \left( \iota(\bR(\A)) \cap \Omega_\eta \right) \bigr) = 
m_H \bigl( \iota(\bH(\A)) \cap \Omega_\eta \bigr) \cdot m_R \bigl( \iota(\bR(\A)) \cap \Omega_\eta \bigr).
\]
On the other hand, by \cite[\S2.3]{EMMV}, we have 
\[
m_G \left( \iota(\bG(\A)) \cap \Omega_{c \eta} \right) \ll_{\eta} m_G \left( \iota(\bG(\A)) \cap \Omega_{\eta} \right).
\]
Altogether, it follows that
\begin{align*}
\vol(Y) 	&\ll m_G \bigl( \iota(\bG(\A)) \cap \Omega_{c \eta_0} \bigr)^{-1} \\
			&\ll m_H \bigl( \iota(\bH(\A)) \cap \Omega_{\eta_0} \bigr)^{-1} \cdot m_R \bigl( \iota(\bR(\A)) \cap \Omega_{\eta_0} \bigr)^{-1}\\ 
			&= \vol(Y_H) \vol(Y_R). \medskip
\end{align*}

To conclude the upper bound estimate, it thus suffices to show that 
\[
\height(\G)\ll (\vol(Y_H) \vol(Y_R))^\star.
\] 
To see this first note that since $\gfrak=\fh\oplus\fr$, we have $\height(\G)\ll(\height(\bH)\height(\bR))^\star$.
Now by Lemma~\ref{lem:ht-vol-unip}, we have $\height(\bR)\ll \vol(Y_R)$. Moreover, by~\cite[App.~B]{EMMV}
we have $\height(\bH)^\star\ll \vol(Y_H)$. The claim follows.

\medskip

{\em The lower bound.}\
For the lower bound estimate, we will use notation from the proof of Lemma~\ref{lem:height-g-g-H};
in particular, $\G'=\iota(\G)$, $\bH'$ is a Levi subgroup of $\G'$ and $\bR'=\rad_u(\G')$.
Recall from the proof of Lemma~\ref{lem:height-g-g-H} that there exists some 
$\delta=(\tfrac{a_{ij}}{b_{ij}})\in\SL_N(\Q)$ with $|a_{ij}|,|b_{ij}|\ll \height(\bR)^\star$ 
and a block upper triangular parabolic subgroup ${\bf P}_0\subset \SL_N$ so that 
\[
\delta\G'\delta^{-1}\subset{\bf P}_0\quad\text{and}\quad\delta\bR'\delta^{-1}\subset\rad_u({\bf P}_0).
\]
Recall also that $\hfrak'=\Lie(\bH')$, and that ${\bf Q}_0=[{\bf P}_0, {\bf P}_0]\rad_u({\bf P}_0)$. We define ${\bf M}$ to be the block diagonal Levi subgroup of ${\bf Q}_0$.

Apply Proposition~\ref{prop:LiesmallLevi} with $\Ad(\delta)\hfrak'\subset\Lie({\bf Q}_0)$.
Therefore, there exists some Levi subgroup ${\bf M}'\subset {\bf Q}_0$ so that
\[
\delta\bH'\delta^{-1}\subset{\bf M}'\quad\text{and}\quad\height({\bf M}')\ll \height(\G)^\star.
\]

Let $\Bcal=\{v_1,\ldots, v_d\}$ be a $\Z$-basis for $\Lie(\rad_u({\bf P}_0))\cap\fsl_N(\Z)$ with $\|v_i\|\ll 1$.
Similarly, let $\mathcal C=\{w_1,\ldots, w_m\}$ (resp.\ $\mathcal C'=\{w_1',\ldots, w_m'\}$) be $\Z$-bases for
$\Lie({\bf M})\cap\fsl_N(\Z)$, (resp.\ $\Lie({\bf M}')\cap\fsl_N(\Z)$) with $\|w_i\|\ll 1$ and $\|w'_i\|\ll \height(\G)^\star$.

Recall that any two Levi subgroups of ${\bf Q}_0$ are conjugate to each other by an element in $\rad_u({\bf P}_0)$.
Writing these equations (in the Lie algebra) in the bases $\mathcal C$ and $\mathcal C'$ in terms of $\mathcal B$ we get the following. 
There exists some $u=(u_{ij})\in\rad_u({\bf P}_0)(\Q)$ with $u_{ij}=(\tfrac{c_{ij}}{d_{ij}})$ and $|c_{ij}|,|d_{ij}|\ll \height(\G)^\star$ 
so that $u{\bf M}'u^{-1}={\bf M}$.

Altogether, there exist some $\hat\delta=({\hat a_{ij}}/{\hat b_{ij}})\in\SL_N(\Q)$ with $|\hat a_{ij}|,|\hat b_{ij}|\ll \height(\G)^\star$ 
so that 
\be\label{eq:Hhat-Rhat}
\hat\delta\bH'\hat\delta^{-1}\subset{\bf M}\quad\text{and}\quad\hat\delta\bR'\hat\delta^{-1}\subset\rad_u({\bf P}_0).
\ee
Put $\hat G=\hat\delta \iota(\G(\A))\hat\delta^{-1}$, and define $\hat H$, $\hat R$ similarly. 
Having in mind our notations $G_p=\iota(\G(\Q_p))$, etc., we write similarly $\hat G_p=\hat\delta \iota(\G(\Q_p))\hat\delta^{-1}$, etc.

Let $h\in\SL_N(\Z_p)\cap {\bf Q}_0$. We can write $h=h_0h_1$ where 
$h_0\in\SL_N(\Z_p)\cap {\bf M}$ and $h_1\in\SL_N(\Z_p)\cap\rad_u({\bf P}_0)$.
In consequence, we have 
\be\label{eq:Gp=HpRp}
\hat G_p\cap\SL_n(\Z_p)=(\hat H_p\cap\SL_N(\Z_p))(\hat R_p\cap\SL_N(\Z_p))
\ee
for all primes $p$. Conjugating~\eqref{eq:Gp=HpRp} by $\hat\delta^{-1}$, we get
\[
G_p\cap\hat\delta_p^{-1}\SL_N(\Z_p)\hat\delta_p=(H_p\cap\hat\delta_p^{-1}\SL_N(\Z_p)\hat\delta_p)(R_p\cap\hat\delta_p^{-1}\SL_N(\Z_p)\hat\delta_p).
\]
In particular, the image, $I_p$, of the product map from $(H_p\cap\SL_N(\Z_p))\times (R_p\cap\SL_N(\Z_p))$
into $G_p$ contains $G_p\cap\SL_N(\Z_p)\cap \hat\delta_p^{-1}\SL_N(\Z_p)\hat\delta_p$ for all primes $p$.
Therefore, 
\be\label{eq:lowerbound-p}
m_{G_p}(I_p)\geq {m_{G_p}(G_p\cap\SL_n(\Z_p))}/{J_p}
\ee
where $J_p=[\SL_N(Z_p):\SL_N(\Z_p)\cap \hat\delta_p^{-1}\SL_N(\Z_p)\hat\delta_p]$ for all primes $p$.

Since $\hat\delta=({\hat a_{ij}}/{\hat b_{ij}})\in\SL_N(\Q)$ with $|\hat a_{ij}|,|\hat b_{ij}|\ll \height(\G)^\star$, we have
\be\label{eq:prod-Jp}
\textstyle\prod_p J_p\ll \height(\G)^\star
\ee

We also need an estimate for the real place. 
Let $0<\eta\leq \eta_0$ be a constant which will be determined in the following.
Suppose $g \in \iota(\bG(\A)) \cap \Omega_{\eta}$ and write 
$g = (g_\infty, (g_p))$. 
By definition, $g_\infty = \exp w$ for some $w \in \fg' \otimes \R$ with $\|w\| \leq \eta$. 
By Proposition \ref{LiesmallLevi} and our choice of $\fh'$, we can write $w = w_{\fh'} + w_{\fr'}$ 
with $w_{\fh'} \in \fh' \otimes \R$, $w_{\fr'} \in \fr' \otimes \R$ and $\|w_{\fh'}\|, \|w_{\fr'}\| \ll \height(\G)^{\star}\eta^{\star}$. 
We pick $\eta$ in such a way $\eta\ll\height(\G)^{-\star}$, so that the above implies 
\[
\|w_{\fh'}\|, \|w_{\fr'}\| \leq \epsilon\eta_0
\]
for some $\epsilon$ which will be specified momentarily.

Using the Baker-Campbell-Hausdorff formula and the fact that $\fr'$ is an ideal of $\fg'$, 
we see that the Levi component $(g_\infty)_{\iota(\bH(\bbr))}$ of $g_\infty$ is just $\exp(w_{\fh'})$. 
Therefore, if $\epsilon\ll1$ is chosen small enough, we get that $(g_\infty)_{\iota(\bH(\bbr))} \in \Xi_{\eta_0}$ and 
$(g_\infty)_{\iota(\bR(\bbr))}\in\Xi_{\eta_0}$.
In consequence, we we have 
\be\label{eq:lowerbound-infty}
m_{G_\infty}(G_\infty\cap\Xi_{\eta_0})\ll \height(\G)^\star m_{H_\infty}(H_\infty\cap\Xi_{\eta_0})m_{R_\infty}(R_\infty\cap\Xi_{\eta_0}).
\ee

Altogether, we have
\begin{align*}
\vol(Y_H) \vol(Y_R)&=m_H(\iota(\bH(\A))\cap\Omega)^{-1}m_H(\iota(\bR(\A))\cap\Omega)^{-1}\\
{}^{\text{\eqref{eq:lowerbound-infty}}\leadsto}&\ll\height(\G)^{\star}m_{G_\infty}(G_\infty\cap\Xi_{\eta_0})^{-1}\textstyle\prod_p\bigl(m_{G_p}(I_p)\bigr)^{-1}\\
{}^{\text{\eqref{eq:lowerbound-p}}\leadsto}&\ll \height(\G)^\star\vol(Y)\textstyle\prod_pJ_p\\
{}^{\text{\eqref{eq:prod-Jp}}\leadsto}&\ll \height(\G)^\star\vol(Y)\\
{}^{\text{\eqref{eq:def-discY}}\leadsto}&\ll \disc(Y)^\star.
\end{align*}

This implies the lower bound estimate and finishes the proof.
\end{proof}

\section{Proof of Theorem~\ref{thm:adelic-red}}\label{sec:proof}
We now combine the results from previous sections 
to complete the proof of Theorem~\ref{thm:adelic-red} --- 
the idea is to use the effective Levi decomposition of~\S\ref{sec:eff-levi} 
to reduce the problem to the case of semisimple and unipotent groups.

\subsection{Semisimple case}\label{sec:semisimple}

In the next paragraphs, we prove (a slightly finer version of) Theorem~\ref{thm:adelic-red} under the assumption that $\G$ is semisimple.
Therefore, until the end of \S \ref{sec:proof-semisimple}, $\G$ is assumed to be a connected, simply connected, semisimple group.
Under these assumptions the following was proved in~\cite{EMMV}.

\begin{propo}\label{prop:good-place}
There exists a prime $p$ and a parahoric subgroup $K_p$ of $\G(\Q_p)$ so that the following hold.
\begin{enumerate}
\item $p \ll \bigl(\log(\vol(Y))\bigr)^{2}.$
\item $\G$ is quasi-split over~$\Q_p$ and split over~$\widehat{\Q_p}$, the maximal unramified extension of $\Q_p$; 
further,~$K_p$ is a hyperspecial subgroup of~$\G(\Q_p)$.
\item Let $\Gfrak_p$ be the smooth $\Z_p$-group scheme associated to $K_p$ by Bruhat-Tits theory (see~\ref{sec:BrTits}). 
The map $\iota$ extends to a closed immersion from $\Gfrak_p$ to $\SL_N$. 
\item  There exists a homomorphism $\theta_p:\SL_2\rightarrow \Gfrak_p$ 
so that the projection of $\theta_p(\SL_2(\Q_p))$ into each $\Q_p$-almost simple factor of $\G(\Q_p)$ is nontrivial. 
\end{enumerate}

\end{propo} 

\begin{proof}
Parts~(1) and (2) are proved in~\cite[\S5.11]{EMMV};
part~(3) is proved in~\cite[\S6.1]{EMMV}; part~(4) is proved in~\cite[\S6.7]{EMMV}.
\end{proof}

Let $p$ be as in Proposition~\ref{prop:good-place} and let $\theta_p$ be as in Proposition~\ref{prop:good-place}(4).
We define the one-parameter unipotent subgroup~
\[
\text{$u:\Q_p\to\theta_p(\SL_2(\Q_p))$ by }
\; u(t)=\theta_p\left(\begin{pmatrix} 1&t\\0&1\end{pmatrix}\right).
\]
Note that in view of Proposition~\ref{prop:good-place}(2) and~(3) we have
\be\label{eq:u(t)-op-norm}
|\iota(u(t))|\ll (1+|t|_p)^{\star}.
\ee

\subsection{Property $\tau$}\label{sec:prop-tau} 
Recall that $\G$ is quasi-split over $\Q_p$; 
in particular, all of the almost simple factors of~$\G$ are $\Q_p$-isotropic.
Our proof relies on the uniform spectral gap; this deep input has been obtained in a series of papers~\cite{Kazh, O, Sel, JL, BS, LC, GMO}. In particular,
\begin{itemize} 
\item using \cite[Thm.~1.1--1.2]{O} when $\G(F_w)$ has property $(T)$, and 
\item applying property $(\tau)$ in the strong form, see~\cite{LC}, \cite{GMO}, and also~\cite[\S4]{EMMV}, in the general case,
\end{itemize}
we have the following.

\begin{thm}[Property ($\tau$)]\label{thm:unif-prop-tau}
Let $\sigma$ be the probability $\G(\A)$-invariant measure on $\G(\A)/\G(\Q)$.
The representation of $\SL_2(\Q_p)$ via~$\theta_p$ on 
\[
L_0^2(\sigma) :=\bigl\{f\in L^2(\G(\A)/\G(\Q),\sigma):\textstyle\int f\operatorname{d}\!\sigma=0\bigr\}
\]  
is $1/\temp$-tempered. In other words, the matrix coefficients of the $\temp$-fold tensor product
are in $L^{2+\epsilon}(\SL_2(\Q_p))$ for all~$\epsilon>0$.
\end{thm}

It follows from the above theorem that for any $f_1,f_2\in C_c^\infty(\G(\A)/\G(\Q))$ we have
\be\label{eq:sobnormmc}\Big| \langle u(t) f_1, f_2 \rangle_{\sigma} - 
	\int f_1\operatorname{d}\!\sigma \int\bar f_2\operatorname{d}\!\sigma \Big|
 \ll (1+|t|_p)^{-1/2\temp} \Sob(f_1) \Sob(f_2),
\ee
where $\Sob$ is a certain Sobolev norm. We refer to~\cite[App.~A]{EMMV} for the definition and the discussion of the Sobolev norm $\Sob$.

\medskip

Let $\eta>0$ and put $\Xi_{\G,\eta} := \exp(B_{\gfrak_\infty}(\eta)) \subset \G(\R)$. 
For every prime $q$, we set $K_q := \iota^{-1}(\SL_N(\Z_q)) \subset \G(\Q_q)$.
Put $\Omega_{\G,\eta} := \Xi_{\G,\eta}\times \prod_{\places_f}K_q \subset \bG(\A)$.
We set $\Omega_\G = \Omega_{\G,\eta_0}$, see \S\ref{sec:X-eta}. 

\begin{thm}[Semisimple version of Theorem \ref{thm:adelic-red}]\label{sec:proof-semisimple}
There exists some $\consta\label{k:ss-exp-inj-r}\label{k:ss-exp-vol}$ depending only on $N$, and for any datum $(\bG, \iota)$ with $\bG$ semisimple, there exists some $p \in \places_f$ with
\[
p \ll \bigl(\log(\vol(Y))\bigr)^{2},
\]
so that the following holds. For any $g\in\G(\A)$, there exists some $\gamma\in\G(\Q)$ such that 
$g\gamma=h_1hh_2$, where
$h_1, h_2 \in\Omega_{\G,\eta}$ and $h\in \G(\bbq_p)$ with 
\[
|\iota(h)|\ll \height(\iota(g))^{-\ref{k:ss-exp-inj-r}} \vol(Y)^{\ref{k:ss-exp-vol}}.
\]
Moreover, the implicit multiplicative constants depend only on $N$. 
\end{thm}

\begin{proof}
Recall that $m$ is the Haar measure on $G$ which projects to $\mu_Y$. 
Let $\lambda$ be the Haar measure on $\G(\A)$ so that $\iota_*\lambda=m$.
By~\cite[\S5.9]{EMMV} there exists some $M\geq 1$ depending only on $\dim\G$ so that 
\be\label{eq:la-M-ss-adelic}
1/ M\leq \lambda(\G(\A)/\G(\Q))\leq M.
\ee
In view of the definition of $\vol(Y)$, this implies that $\vol(Y)\asymp\lambda(\Omega_\G)^{-1}$.

Let $\eta$ be a positive constant.
For any $g\in\G(\A)$ put $[g]=g\G(\Q)$; assume $\iota([g])\in X_{\injr}$.
We claim that if $h,h'\in \Omega_{\G,\eta}$ are so that $h[g]=h'[g]$,
then $h^{-1}h'\in{\bf Z}(\Q)$, where ${\bf Z}:={\rm Z}(\G)$ denotes the center of $\bG$.
To see this, apply $\iota$ to the equation $h[g]=h'[g]$. Using the definition of 
$X_\eta$ and the fact $\iota(\Omega_{\G,\eta})\subset\Omega_\eta$, we get that $\iota(h)=\iota(h')$.
Hence, $h^{-1}h'\in{\bf Z}(\A)$; moreover $h^{-1}h'=g^{-1}h^{-1}h'g\in\G(\Q)$. Thus
$h^{-1}h'\in{\bf Z}(\Q)$ as claimed.
This claim in particular implies that $\pi_{[g]}: \Omega_{\G,\eta}\to\G(\A)/\G(\Q)$ defined by $\pi_{[g]}(h):=h[g]$ 
is at most $\#{\bf Z}(\Q)$-to-one on $\Omega_{\G,\eta}$.

By~\cite[App.~A]{EMMV}, there exists a function $f \in C_c^{\infty}(\G(\A))$ with the following properties: 
\begin{itemize}
\item $0\leq f\leq 1$,
\item for all $h\not\in\Omega_{\G,\eta}$ we have $f(h)=0$ and for all $h\in\Omega_{\G,\eta/2}$ 
we have $f(h)=1$,
\item $\Sob(f)\ll \eta^{-\star}$.
\end{itemize}
For every $g\in\G(\A)$ with $\iota([g])\in X_{\injr}$, define $f_{[g]}\in C_c^{\infty}(\G(\A)/\G(\Q))$ as follows. 
If $[g']\in\pi_{[g]}(\Omega_{\G,\eta})$, put $f_{[g]}([g'])=\sum_{\pi_{[g]}(h)=[g']} f(h)$; if $[g']\not\in\pi_{[g]}(\G_\eta)$, define 
$f([g'])=0$. Then 
\begin{enumerate}
\item $0\leq f_{[g]}\leq \#{\bf Z}(\Q)\ll1$, 
\item $f([g'])=0$ for all $[g']\not\in\pi_{[g]}(\Omega_{\G,\eta})$ and $f_{[g]}([g'])\geq 1$ for all $[g']\in\pi_{[g]}(\Omega_{\G,\eta/2})$, 
\item $\Sob(f_{[g]})\ll \eta^{-\star}$.
\end{enumerate}

Recall the measure $\sigma$ from Theorem~\ref{thm:unif-prop-tau}. By~\eqref{eq:la-M-ss-adelic}, we have that
$
\int f_{[g]}\operatorname{d}\!\sigma\asymp\int f_{[g]}\operatorname{d}\!\lambda.
$
The set $\Omega_{\G}$ can be covered by $\ll\eta^{-\star}$ translates among $\{h\Omega_{\G,\eta/2}: h\in\G(\A)\}$. Since $\lambda$ is $\bG(\A)$-invariant, this implies that $\lambda(\Omega_\bG) \ll \eta^{-\star} \lambda(\Omega_{\bG,\eta/2})$. Thus,
\be\label{eq:int-f-eta-GA-GQ}
\int f_{[g]}\operatorname{d}\!\sigma\asymp\int f_{[g]}\operatorname{d}\!\lambda \geq  \lambda(\Omega_{\bG, \eta/2}) \gg\eta^{\star} \vol(Y)^{-1};
\ee
here, we used properties (1) and~(2) of $f_{[g]}$, and the fact $\vol(Y)\asymp\lambda(\Omega_\G)^{-1}$.

Apply~\eqref{eq:sobnormmc} with $f_1=f_{[e]}$ and $f_2=f_{[g]}$. Using property~(3) of $f_1$ and $f_2$, we get that 
\begin{equation}\label{eq:sobnormmc-pf-ss}\Big| \langle u(t) f_1, f_2 \rangle_{\sigma} - 
	\int f_1\operatorname{d}\!\sigma \int f_2\operatorname{d}\!\sigma \Big|\\ 
 \ll (1+|t|_p)^{-1/2\temp} \eta^{-\star}.
\end{equation}

We get from~\eqref{eq:sobnormmc-pf-ss} and~\eqref{eq:int-f-eta-GA-GQ} (which also holds for $f_{1}$)
that if $|t|_p\gg \vol(Y)^\star\eta^{-\star}$, then 
\be\label{eq:inn-prod-n0}
\langle u(t) f_1, f_2 \rangle_{\sigma}\neq 0.
\ee
This implies in particular that if $|t|_p\gg\vol(Y)^\star\eta^{-\star}$, then the following holds. 
There exist $h_1, h_2 \in\G(\A)$ so that $f_1([h_1])\neq 0$, $f_2([h_2^{-1}g])\neq0$, and
\be\label{eq-ut-h1hh2}
u(t)h_1\G(\Q)=h_2^{-1}g \G(\Q).
\ee
In view of the fact that $\Omega_{\G,\eta}=\Omega_{\G,\eta}^{-1}$, it follows from the above and 
property~(2) that $h_i\in \Omega_{\bG,\eta}$.

Finally, we choose $t$ so that  \eqref{eq:inn-prod-n0} holds while $|t|_p\asymp\vol(Y)^\star\eta^{-\star}$. In this way, by~\eqref{eq:u(t)-op-norm} we have $|\iota(u(t))|\ll (1+|t|_p)^{\star} \ll \vol(Y)^\star\eta^{-\star}$. 
In view of~\eqref{eq-ut-h1hh2}, by taking $h_1$ and $h_2$ as above and $h=u(t)$, the proof of Theorem~\ref{sec:proof-semisimple} is complete. 
\end{proof}

Before proceeding to the proof of general case, we need the following

\begin{lemma}\label{lem:unip-case}
There exists some $\consta\label{k:unip-exp}$ so that the following holds. 
Let $\bR$ be a unipotent $\Q$-group, given with an embedding $\iota: \bR \to \SL_N$. 
Let $S\subset\places$ be a finite set of places containing the infinite place;
put $p_S:=\max\{p\in S \cap \places_f \}$. 
Let $v\in S$. For any $g\in\bR(\A)$, there exists some $\gamma\in\bR(\Q)$ so that
\[
\iota(g\gamma)=(h_S, (h_q)_{q\not\in S})\in \SL_N(\bbq_S)\times \prod_{q\not\in S}\SL_N(\Z_q),
\]
$|h_v|\ll p_S^{\ref{k:unip-exp}}\height(\bR)^{\ref{k:unip-exp}}$, and for every $w \in S - \{v\}$, 
we have $|h_{w}|\ll p_S^{\ref{k:unip-exp}}$. 
\end{lemma}

\begin{proof}
The proof is, mutatis mutandis, part of the proof of Lemma~\ref{lem:ht-vol-unip}. We briefly recall the argument for the convenience of the reader. 
For every prime $q$, put 
\[
C_q=\iota^{-1}(\iota(\bR(\Q_q))\cap\SL_N(\Z_q)).
\] 
By the strong approximation theorem for unipotent groups, we have
\[
\bR(\A)=\bigl(\bR(\bbq_S)\times\textstyle\prod_{q\not\in S}C_q\bigr)\bR(\Q).
\] 
Hence, there exists some $\gamma_0\in\bR(\Q)$ so that 
\[
g\gamma_0=(\hat g_S,(\hat g_q)_{q\not\in S})\in\bR(\bbq_S)\times\textstyle\prod_{q\not\in S}C_q.
\]

Fixing a $\Z_S$-basis for $\fr(\Q_S) \cap\fsl_N(\Z_S)$, we have the following. 
There exists some $\hat\gamma_1\in\iota^{-1}(\exp(\fr(\Q_S) \cap N!\fsl_N(\Z_S)))$ 
so that $h_S=\iota(\hat g_S\hat\gamma_1)$ satisfies
\[
|h_v|\ll p_S^\star\height(\bR)^\star\quad \text{and} \quad|h_w|\ll p_S^\star \textrm{ for } w \in S-\{v\}.
\] 
Note that $\iota(\hat\gamma_1) \in \iota(\bR(\Q_S)) \cap \SL_N(\Z_S)$, hence $\hat\gamma_1 \in \bR(\Q)$. 

Let $\gamma_1$ be the diagonal embedding of $\hat\gamma_1$ in $\bR(\A)$. Then 
since $\iota(\hat\gamma_1)\in \SL_N(\Z_S)$, we get that 
\[
g\gamma_0\gamma_1=(\hat g_S,(\hat g_q)_{q\not\in S})\gamma_1=(\hat g_S\hat\gamma_1, (\tilde g_q)_{q\not\in S})\in\bR(\bbq_S)\times\textstyle\prod_{q\not\in S}C_q. 
\] 
The claim thus follows with $\gamma=\gamma_0\gamma_1$.
\end{proof}

\subsection{Proof of Theorem~\ref{thm:adelic-red}}\label{sec:proof-general}
Let $g\in\G(\A)$ and write $g=g_Hg_R$ where $g_H\in\bH(\A)$ and $g_R\in\bR(\A)$; 
recall that ${\rm pr}_H(g) = g_H$.

First, we apply Theorem~\ref{sec:proof-semisimple}, i.e.\ the semisimple case, to the pair $(\bH, \iota_{|_H})$. 
In view of Lemma~\ref{lem:ht-injr}, we have $\iota(g_H\G(\Q))\in X_\eta$ for $\eta :=\ref{eq:ht-injr}\height(\iota(g_H))^{-\ref{eq:ht-injr}}$. 
Thus, there exist some $\gamma_0\in\bH(\Q)$ and some 
$
p\ll \bigl(\log\vol(Y_H)\bigr)^2
$
so that the following holds.
There are $h \in\bH(\bbq_p)$ and $h_1, h_2 \in \Omega_\bH \subset \Omega_\bG$ 
such that $g_H\gamma_0= h_1 h h_2$ and 
\[
|\iota(h)|\ll {\eta}^{-\ref{k:ss-exp-vol}}\vol(Y_H)^{\ref{k:ss-exp-vol}}.
\]
This estimate implies that 
\begin{align}
\notag|\iota(h)|&\ll {\height(\iota(g_H))}^{\star}\vol(Y_H)^{\star} &&\text{since $\eta =\ref{eq:ht-injr}\height(\iota(g_H))^{-\ref{eq:ht-injr}}$}\\
\notag&\ll \height(\G)^\star\height(\iota(g))^\star\vol(Y_H)^\star&&\text{by Lemma~\ref{lem:height-g-g-H}}\\
\notag&\ll \height(\G)^\star\height(\iota(g))^\star\disc(Y)^\star&&\text{by Prop.~\ref{prop:vol-comp}}\\
\label{eq:ss-est-pf}&\ll\height(\iota(g))^\star\disc(Y)^\star&&\text{by~\eqref{eq:def-discY}}.
\end{align}

Also note that by Proposition~\ref{prop:vol-comp} we have
\be\label{eq:p-est}
p\ll \bigl(\log\vol(Y_H)\bigr)^2\ll \bigl(\log\height(Y)\bigr)^\star.
\ee

Apply Lemma~\ref{lem:unip-case} with  $S=\{\infty, p\}$ and $v=p$ to the element $\gamma_0^{-1}g_R\gamma_0\in\bR(\A)$; we get the following.
There exists some $\gamma_1\in\bR(\Q)$ for which
\begin{enumerate}[label=(\alph*)]
\item $\iota(\gamma_0^{-1}g_R\gamma_0\gamma_1)\in \SL_N(\Z_q)$ for all primes $q\neq p$,
\item $|\iota(\gamma_0^{-1}g_R\gamma_0\gamma_1)_\infty|\ll p^\star$, and 
\item $|\iota(\gamma_0^{-1}g_R\gamma_0\gamma_1)_p|\ll p^\star\height(\bR)^\star\ll p^\star\height(\G)^\star$.
\end{enumerate}
\medskip

Set $\gamma=\gamma_0\gamma_1\in\G(\Q)$. Let us write
\[
(\hat g_\infty, \hat g_p, (\hat g_q)_{q\not\in S}):=\iota(g\gamma)=\iota(g_H\gamma_0)\iota(\gamma_0^{-1}g_R\gamma_0\gamma_1).
\]
The above estimates then imply that
\begin{enumerate}
\item By (a) and $h_i\in \Omega_\G$, $i=1,2$, we have $\hat g_q\in \SL_N(\Z_q)$ for all primes $q\neq p$.
\item By (b) and $h_i\in \Omega_\G$, $i=1,2$, we have 
\begin{align*}
|\hat g_\infty|&\ll |\iota(h_1 h h_2)_\infty| \cdot |\iota(\gamma_0^{-1}g_R\gamma_0\gamma_1)_\infty|\ll p^\star\\
&\ll (\log\disc (Y))^\star&&\text{by~\eqref{eq:p-est}}
\end{align*}
\item For the prime $p$, we have 
\begin{align*}
|\hat g_p|&\ll |\iota(h_1 h h_2)_p||\iota(\gamma_0^{-1}g_R\gamma_0\gamma_1)_p|\\
&\ll \height(\iota(g))^\star\disc(Y)^\star p^\star\height(\G)^\star&&\text{by~\eqref{eq:ss-est-pf} and~(c)}\\
&\ll \height(\iota(g))^\star \disc(Y)^\star \height(\G)^\star &&\text{by~\eqref{eq:p-est}}\\
&\ll \height(\iota(g))^\star\disc(Y)^\star&&\text{by~\eqref{eq:def-discY}}.
\end{align*}
\end{enumerate}

The proof is complete.
\qed


\section{$S$-Arithmetic Quotients}\label{sec:local}

In this section, we discuss some implications of the statement and the proof of Theorem~\ref{thm:adelic-red} in the local setting. The main results are stated in Theorem~\ref{thm:eff-diam-local-semisimple} which deals with the case of semisimple groups and Theorems~\ref{thm:eff-strong-app-direct} and~\ref{thm:eff-strong-app} which can be thought of as {\em effective} versions of the strong approximation theorem.

\subsection{The setup}\label{sec:local-notation}

Let $\lcagr \subset \SL_d$ be a $\Q$-group so that $\rad(\lcagr)=\rad_u(\lcagr)$. 
Let $S\subset\places$ be a finite set of places containing the infinite place.
Define 
\[
\lcgr :=\prod_{v\in S}\lcagr(\bbq_v)\quad\text{and}\quad\fl :=\oplus_{v\in S}\fl_v, 
\] 
where $\fl_v :=\Lie(\lcagr)(\bbq_v)$.

Let $\bR=\rad_u(\lcagr)$. 
Fix a Levi subgroup $\bH$ of $\lcagr$ so that $\height(\bH)\ll\height(\lcagr)^\star$, see Proposition~\ref{prop:LiesmallLevi}.
We let $\tilde\bH$ denote the simply connected covering of $\bH$.
Put $\tlg=\tilde\bH \ltimes \bR$, where the action of $\tilde\bH$ on $\bR$
factors through the action of $\bH$ via the natural covering map $\pi':\tilde\bH\to\bH$. 
By the construction of $\tlg$, $\pi'$ extends to an epimorphism $\pi:\tlg\to\lcagr$ with finite central kernel, given by $\pi(g)=\pi(g_{\tilde \bH}g_{\bR})=\pi'(g_{\tilde \bH})g_\bR$, where $g_{\tilde\bH}\in\tilde\bH$ and $g_{\bR}\in\bR$. 

Let $\tlc:=\pi(\tlg(\Q_S))$; then $\tlc$ is a normal subgroup of $\lcgr$ and $\lcgr/\tlc$ is a finite abelian group --- it is worth mentioning that this finite group can be identified with a subgroup of $\prod_S \rH^1(\Q_v,Z(\tilde\bH))$.

\subsection{Two notions of complexity}\label{sec:S-arith-ss}\label{sec:vol-vol}
For every $q\in\places_f$ put $K_q=\pi^{-1}(\SL_d(\Z_q))$. Define the subgroups $\Delta$ and $\Gamma$ of $\tlg(\Q_S)$ as follows:
\be\label{eq:def-Delta}
\Delta:=\text{the projection of $\tlg(\Q)\cap(\tlg(\Q_S) \times \textstyle\prod_{q\not\in S}K_q)$ to $\tlg(\Q_S)$},
\ee
and $\Gamma:=\pi^{-1}(\SL_d(\Z_S))$. 
Note that $\Delta$ is a normal subgroup of $\Gamma$; moreover, both $\Delta$ and $\Gamma$ are lattices in $\tlg(\Q_S)$.

Put $Z :=\pi\bigl(\tlg(\A)/\tlg(\Q)\bigr)$. 
Similarly define 
\[
\hat Z:=\pi(\tlg(\Q_S)/\Delta)=\tlc/\tlc\cap\SL_d(\Z_S)=\tlc/\pi(\Gamma).
\]

As was done in \S\ref{sec:S-arith-intro}, we define $\vol(\hat Z)=m_S(\tlc\cap\Omega_S)^{-1}$ where $\Omega_S=\Xi_{\eta_0}\times \prod_{q\in S - \{\infty\}}\SL_d(\Z_q)$
and $m_S$ is a Haar measure on $\tlg(\Q_S)$ normalized so that $m_S(\hat Z)=1$. Here and in what follows, we abuse the notation and denote $\pi_*\nu$ simply by $\nu$, for any measure $\nu$. 

We also put $\height(\hat Z)=\max\{\height(\lcagr), \vol(\hat Z)\}$.

\begin{propo} \label{prop:vol-vol}
There exist $\consta\label{k:volZ-volZhat}$, $\consta\label{k:volZ-volZhat-up}$, and
$\consta\label{k:volZ-volZhat-M}$ so that for all $\bL$ as in \ref{sec:local-notation} with $\vol(Z)\gg1$, we have the following.
\begin{enumerate}
\item $\ref{k:volZ-volZhat-M}^{-1}\height(Z)^{\ref{k:volZ-volZhat}}\leq\height(\hat Z)\leq \ref{k:volZ-volZhat-M}\height(Z)^{\ref{k:volZ-volZhat-up}}$;
\item If $\tlg$ is semisimple or unipotent, then  
\[
\ref{k:volZ-volZhat-M}^{-1}\vol(Z)^{\ref{k:volZ-volZhat}}\leq\vol(\hat Z)\leq \ref{k:volZ-volZhat-M}\vol(Z)^{\ref{k:volZ-volZhat-up}}
\]
\end{enumerate}
\end{propo}

\begin{proof}
We first prove part~(2) above.

First note that if $\tlg$ is unipotent, then $\tlg=\lcagr$ and the same argument as in Lemma~\ref{lem:ht-vol-unip} 
implies that $\height(\tlg)^\star\ll\vol(\hat Z)\ll \height(\tlg)$. The claim in this case follows from this and Lemma~\ref{lem:ht-vol-unip}.

We now assume that $\tlg$ is semisimple. 
In this case we will actually prove
\be\label{eq:vol-vol-semi}
\ref{k:volZ-volZhat-M}^{-1}\vol(Z)^{\ref{k:volZ-volZhat}}\leq \vol(\hat Z)\leq \ref{k:volZ-volZhat-M}\vol(Z)
\ee
when $\vol(Z)$ is large enough.

Let $\lambda$ denote the Haar measure on $\tlg(\A)$ normalized so that $\lambda(Z)=1$.
By~\cite[\S5.9]{EMMV} there exists\footnote{The discussion in~\cite[\S5.9]{EMMV} assumes that $\tlg$ is $\Q$-almost simple; since $\tlg$ is simply connected and semisimple, we can decompose $\tlg=\tlg_1\cdots\tlg_r$ as a direct product of $\Q$-almost simple factors and apply the argument to each factor separately.} some ${M}\geq 1$ depending only on $\dim\tlg$ so that 
\be\label{eq:la-M}
1/{M}\leq \lambda(\tlg(\A)/\tlg(\Q))\leq {M}.
\ee

Since $\tlg$ is simply connected and $\tlg(\Q_S)$ is not compact, we have
\[
\tlg(\A)= \left(\tlg(\Q_S) \times \textstyle\prod_{q\not\in S}K_q \right)\tlg(\Q).
\]
Write $\lambda=\prod_{\places} \lambda_v$ and set $\lambda_S :=\prod_{S}\lambda_v$. 
In view of the above and the definition of $\Delta$, see~\eqref{eq:def-Delta}, we get the following. 
\be\label{eq:la-laS-1}
\lambda(\tlg(\A)/\tlg(\Q))=\lambda_S(\tlg(\Q_S)/\Delta) \cdot \textstyle\prod_{q\not\in S}\lambda_q(K_q)
\ee
Recall furthermore that 
\be\label{eq:la-laS-2}
\vol(Z)=\lambda(\pi(\tlg(\A))\cap\Omega)^{-1}=\lambda_S(\tlc \cap \Omega_S)^{-1} \cdot \textstyle\prod_{q\not\in S}\lambda_q(\pi(K_q))^{-1}.
\ee
From~\eqref{eq:la-M},~\eqref{eq:la-laS-1}, and~\eqref{eq:la-laS-2} we get that
\be\label{eq:laS-Delta}
\lambda_S(\tlg(\Q_S)/\Delta)=M'\lambda_S(\tlc \cap \Omega_S) \cdot \vol(Z)
\ee 
where $M'\in [1/{M},{M}]$.

We can now make the following computation.
\begin{align}
\label{eq:laS-Gamma}\lambda_S(\tlg(\Q_S)/\Gamma)&=\lambda_S(\tlg(\Q_S)/\Delta) \cdot [\Gamma:\Delta]^{-1}\\
\notag&=M'\lambda_S(\tlc \cap \Omega_S) \cdot \vol(Z) \cdot [\Gamma:\Delta]^{-1}&&\text{by~\eqref{eq:laS-Delta}}
\end{align}

Perhaps by enlarging $M$ to account for the effect of the central kernel of $\pi$, 
we have $\lambda_S(\tlc/\pi(\Gamma))=M'' \lambda_S(\tlg(\Q_S)/\Gamma)$ for some
$M''\in[1/M,M]$. 
Therefore, writing the definition of $\vol(\hat Z)$ in terms of the measure $\lambda_S$, we have
\begin{align}
\label{eq:vol-Zhat}\vol(\hat Z)&={\lambda_S(\tlc/\pi(\Gamma))} \cdot {\lambda_S(\tlc \cap \Omega_S)}^{-1}\\
\notag&=M''\lambda_S(\tlg(\Q_S)/\Gamma) \cdot {\lambda_S(\tlc \cap \Omega_S)}^{-1}\\
\notag&=\hat M\vol(Z) \cdot [\Gamma:\Delta]^{-1}&&\text{by~\eqref{eq:laS-Gamma}} 
\end{align}
where $\hat M\in[1/{M}^{2},{M}^2]$.

We now apply the discussion in~\cite[\S5.12]{EMMV}, see also~\cite{BorPr} and~\cite[Cor.~6.1]{Belo}, with $\Lambda=\Delta$ and $\tilde\Lambda=\Gamma$ --- note that the only role $S$ plays in the argument in~\cite[\S5.12]{EMMV} is for the use of the strong approximation theorem. It is proved in the proposition in~\cite[\S5.12]{EMMV}, see also the intermediate steps (5.10) and (5.13) in loc.\ cit., 
that there exists some $0<\consta\label{k:normal-ind}<1$ such that
\be\label{eq:index-bound}
[\Gamma:\Delta]\leq \vol(Z)^{\ref{k:normal-ind}},
\ee
provided that $\vol(Z)\gg1$.

In consequence,~\eqref{eq:vol-Zhat} and~\eqref{eq:index-bound} imply~\eqref{eq:vol-vol-semi} with $\ref{k:volZ-volZhat}=1-\ref{k:normal-ind}$ and $\ref{k:volZ-volZhat-M}=M^2$; this finishes the proof of~(2). 

\medskip

We now use the estimate in~(2) to prove~(1). First recall our Levi decomposition 
$\tlg=\tilde\bH\bR$; recall also that $\tlg(\Q)=\tilde\bH(\Q)\bR(\Q)$ and $\tlg(\Q_v)=\tilde\bH(\Q_v)\bR(\Q_v)$ for all $v\in\places$.

Define $\Gamma_H=(\pi_{|_{\tilde\bH}})^{-1}(\SL_d(\Z_S))$, and define $\Gamma_R$ similarly.
Following the above notation, put $\hat Z_H=\pi(\tilde\bH(\Q_S)/\Gamma_H)$ and $\hat Z_R=\pi(\bR(\Q_S)/\Gamma_R)$;
also put $\Lambda=\Gamma_H\Gamma_R\subset\Gamma$.

Let $\nu$ be the Haar measure on $\tlg(\Q_S)$ normalized so
that $\nu(\tlg(\Q_S)/\Lambda)=1$; similarly, let $\nu_H$ and $\nu_R$ be Haar measures on $\tilde\bH(\Q_S)$ and $\bR(\Q_S)$
normalized so that $\nu_H(\tilde\bH(\Q_S)/\Gamma_H)=1$ and $\nu_R(\bR(\Q_S)/\Gamma_R)=1$, respectively. 
In view of the product structure of $\Lambda$ and $\tlg(\Q_S)$, 
we may argue as in~\S\ref{sec:prod-str} and get that $\nu$ is given as the product of $\nu_H$ and $\nu_R$.

The above normalizations of $\nu_H$ and $\nu_R$ and the definitions of $\hat Z_H$ and $\hat Z_R$ imply that
$
\vol(\hat Z_H)=\nu_H\bigl(\pi(\tilde\bH(\Q_S))\cap \Omega_S\bigr)^{-1}
$ 
and $\vol(\hat Z_R)=\nu_R\bigl(\pi(\tilde\bR(\Q_S))\cap \Omega_S\bigr)^{-1}$.
Let us put 
\[
\vol_\nu(\hat Z):= \nu\bigl(\pi(\tlg(\Q_S))\cap \Omega_S\bigr)^{-1}.
\]

Using the product structure of $\nu$ again, we may now argue as in the proof of Proposition~\ref{prop:vol-comp} 
and get that
\be\label{eq:vol-nu-prod}
\bigl(\vol(\hat Z_H) \vol(\hat Z_R)\bigr)^{\star} \ll \height_\nu(\hat Z) \ll \bigl(\vol(\hat Z_H) \vol(\hat Z_R)\bigr)^{\star},
\ee
where $\height_\nu(\hat Z)=\max\{\height(\lcagr),\vol_\nu(\hat Z)\}$.

We now compare 
$\vol_\nu(\hat Z)$ and $\vol(\hat Z)$. Using the notation in the proof of Proposition~\ref{prop:vol-comp}, 
see in particular~\eqref{eq:lowerbound-p}, we have
the following.
\be\label{eq:Gamma-prod}
[\Gamma:\Lambda]\leq \prod_{q\not\in S}J_q\leq \prod_{\places}J_q\ll\height(\lcagr)^\star;
\ee
the first inequality follows from the definition of $\Lambda$, $\Gamma$, and $J_p$, 
the second inequality follows since $J_q\geq 1$ for all $q$, and the third inequality is~\eqref{eq:prod-Jp}.

Recall that $m_S$ denotes the Haar measure on $\tlg(\Q_S)$ normalized so that
$m_S(\tlg(\Q_S)/\Gamma)=1$.  We have
\[
\vol(\hat Z)=m_S\bigl(\pi(\tlg(\Q_S))\cap \Omega_S\bigr)^{-1}=\nu\bigl(\pi(\tlg(\Q_S))\cap \Omega_S\bigr)^{-1}[\Gamma:\Lambda]^{-1}.
\]
This, together with~\eqref{eq:Gamma-prod}, implies that
\[
\vol_\nu(\hat Z)\height(\lcagr)^{-\star}\ll \vol(\hat Z)\leq \vol_\nu(\hat Z),
\]
which in turn gives
\be\label{eq:vol-Zhat-nu}
\height_\nu(\hat Z)^\star\ll\height(\hat Z)\ll\height_\nu(\hat Z).
\ee

Now in view of part~(2), the upper and lower bound in~\eqref{eq:vol-nu-prod} are 
$\asymp\bigl(\vol(Z_H) \vol(Z_R)\bigr)^{\star}$. Moreover, 
Proposition~\ref{prop:vol-comp}(1) gives 
\be\label{eq:vol-comp-use}
\bigl(\vol(Z_H) \vol(Z_R)\bigr)^{\star}\ll\height(Z)\ll \bigl(\vol(Z_H) \vol(Z_R)\bigr)^{\star}.
\ee

The claim in part~(1) follows from~\eqref{eq:vol-nu-prod},~\eqref{eq:vol-Zhat-nu}, and~\eqref{eq:vol-comp-use}.
\end{proof}

We now turn to the consequences of Theorem~\ref{thm:adelic-red} in the $S$-arithmetic setting when applied to the datum $(\tlg,\pi)$.
Recall that we defined
\[
\height_S(g) :=\max \left\{(\textstyle\prod_{S}\|gw\|)^{-1}:0\neq w\in\Z_S^d \right\}
\] 
for any $g\in\SL_d(\Q_S)$.

For any set $S$ of places and any $g \in \SL_d(\Q_S)$ (resp.~$g \in \tlg(\Q_S)$), we write $\tilde g :=(g,(e)_{q\not\in S})\in\SL_d(\A)$ (resp.~$\in \tlg(\A)$). 

\begin{lemma}\label{lem:prod-formula}
For any $g\in\SL_d(\Q_S)$ we have
\[
\height(\tilde g)=\height_S(g).
\]
\end{lemma}

\begin{proof}
This is a consequence of the product formula as we now explicate. 
For every $w\in\Q^d$, let $\bar w$ be a primitive integral vector on $\Q\cdot w$. 
First observe that
\begin{align}
\notag\cfun(\tilde gw)	&=\prod_\places\|\tilde g_v w\|_{v}=\prod_\places\|\tilde g_v \bar w\|_{v}					&&\text{by the product formula}\\
\notag			&=\prod_{S}\|g_v \bar w\|_{v}\prod_{q\not\in S}\|\bar w\|_{q}
					&& \tilde g_q=e,\; q \not\in S\\
\notag			&= \prod_{S}\|g_v \bar w\|_{v}
					&&\text{$\bar w$ is primitive integral}.
\end{align}
This shows that $\height(\tilde g) \leq \height_S(g)$.

To see the reverse inequality, notice that if $w \in \Z_S^d$, then $\|w\|_q \leq 1$ for any $q \notin S$. This implies that
\[
\prod_{S}\|g_v w\|_{v} \geq \prod_{S}\|g_v w\|_{v} \prod_{q \notin S} \|w\| = \prod_{\Sigma}\|\tilde g_v w\|_{v} = \cfun (\tilde g w)
\]
and in turn that $\height_S(g) \leq \height(\tilde g)$. 
\end{proof}

In the following, we use the same notation for the diagonal embedding of elements of $\SL_d(\Q)$ in $\SL_d(\A)$ and in $\SL_d(\Q_S)$;
which embedding is relevant will be indicated by the context.

\begin{thm}\label{thm:eff-strong-app-direct}
There exists $\consta\label{k:local-exp-dir}$ so that the following holds. 
Let the notation be as in \S\ref{sec:vol-vol}.
There exists some 
\[
p\ll \bigl(\log\vol(\hat Z)\bigr)^2
\] 
with the following property. For any $g\in\tlg(\Q_S)$, there exists some $\gamma\in\tlg(\Q)$ so that
$\pi(\gamma)_q\in\SL_d(\Z_{q})$ for all $q\not\in S\cup\{p\}$ and
\[
|\pi(g\gamma)_v|\ll \height_S(\pi(g))^{\ref{k:local-exp-dir}}\disc(\hat Z)^{\ref{k:local-exp-dir}}
\]
for all $v\in S$. 
Moreover, if $p\not\in S$, then 
\[
|\pi(\gamma)_p|\ll \height_S(\pi(g))^{\ref{k:local-exp-dir}}\disc(\hat Z)^{\ref{k:local-exp-dir}}.
\]
\end{thm}

\begin{proof} 
In view of part~(1) of Proposition \ref{prop:vol-vol}, it suffices to prove the above estimates with $\disc(\hat Z)$ replaced by $\disc(Z)$.

In view of Lemma~\ref{lem:prod-formula} and of Theorem~\ref{thm:adelic-red} applied to $(\tlg, \pi)$ and $\tilde{g} \in \tlg(\A)$, there exists some $\gamma\in\tlg(\Q)$
so that $\pi(\tilde {g}\gamma)_v$ satisfies the estimate stated in the theorem for all $v\in S$, and $\pi(\tilde{g}\gamma)_q\in\SL_d(\Z_q)$ for all $q\notin \{\infty, p\}$.
Therefore, $\pi(\gamma)_q\in\SL_d(\Z_{q})$. 
  
Now if $p\not\in S$, then $\pi(\tilde g \gamma)_p=\pi(\gamma)_p$, and the desired estimate follows from Theorem~\ref{thm:adelic-red}. 
\end{proof}

We now state  and prove a reformulation of Theorem~\ref{thm:local-eff-diam-intro} using the above notation.

\begin{thm}\label{thm:eff-diam-local-semisimple}
Let the notation be as in \S\ref{sec:S-arith-ss}; further, assume that 
\begin{enumerate}
\item $\lcagr$ is semisimple, and
\item $\lcgr=\lcagr(\Q_S)$ is not compact. 
\end{enumerate}
There exist $\consta\label{k:local-ss-exp}$ and some $C=C(L)$ so that the following holds. 
For any $g\in \tlg(\Q_S)$ there exists some 
$\delta\in\Delta$, see~\eqref{eq:def-Delta}, so that 
\[
|\pi(g\delta)_v|\leq C\height_S(\pi(g))^{\ref{k:local-ss-exp}}\vol(\hat Z)^{\ref{k:local-ss-exp}}
\]
for all $v\in S$. 
\end{thm}

\begin{proof}
In view of part~(2) of Proposition \ref{prop:vol-vol}, it suffices to prove the above estimates with $\disc(\hat Z)$ replaced by $\disc(Z)$.

As in the proof of Theorem~\ref{thm:eff-strong-app-direct}, we will deduce this theorem from an adelic statement. Let $w\in S$ be a place so that $\lcagr(\Q_w)$ is not compact. 
The required adelic statement here is an analogue of Theorem~\ref{sec:proof-semisimple} where $\G$ in the notation is replaced by $\tlg$
and the place $p$ is replaced by $w$.

Fix a $\Q_w$-representation (with finite kernel) $\theta_w:\SL_2(\Q_w)\to\tlg(\Q_w)$. 
We define the one-parameter unipotent subgroup~
\[
\text{$u:\Q_w\to \theta_w(\SL_2(\Q_w))$ by }
\; u(t)= \theta_w\left(\begin{pmatrix} 1&t\\0&1\end{pmatrix}\right).
\]
Note that 
\be\label{eq:u(t)-op-norm-w}
|u(t)|\ll C_1(1+|t|_w)^{\star}
\ee
for some $C_1$ depending on $\theta_w$ and hence on $L$.

Moreover, it follows from~\cite[Thm.~1.11]{GMO} that for all $f_1,f_2\in C_c^\infty(\tlg(\A)/\tlg(\Q))$ we have
\be\label{eq:sobnormmc-w}\Big| \langle u(t) f_1, f_2 \rangle_{\sigma} - 
	\int f_1\operatorname{d}\!\sigma \int\bar f_2\operatorname{d}\!\sigma \Big|
 \ll (1+|t|_p)^{-1/2\temp} \Sob(f_1) \Sob(f_2),
\ee
where $\Sob$ is a certain Sobolev norm and $\sigma$ is the probability $\tlg(\A)$-invariant measure on 
$\tlg(\A)/\tlg(\Q)$. 

One now repeats the proof of Theorem~\ref{sec:proof-semisimple} replacing~\eqref{eq:u(t)-op-norm} with~\eqref{eq:u(t)-op-norm-w} and~\eqref{eq:sobnormmc} with~\eqref{eq:sobnormmc-w} to get the following. 
For any $g\in \tlg(\A)$, there exist $h_1,h_2 \in  \Omega_{\tlg,\eta}$ and 
$h\in \tlg(\bbq_w)$ with 
\[
|\pi(h)|\leq C\height(\pi(g))^{-\ref{k:ss-exp-inj-r}} \vol(Z)^{\ref{k:ss-exp-vol}}
\]
such that $g\tlg(\Q)=h_1\tilde hh_2\tlg(\Q)$; the constant $C$ depends on $L$ and $d$.

Let $g\in \tlg(\Q_S)$ and apply the above discussion to $\tilde g$. 
Then using the above and Lemma~\ref{lem:prod-formula}, there exists some $h\in\tlg(\bbq_w)$
with 
\[
|\pi(h)|\leq C\height_S(\pi(g))^{-\ref{k:ss-exp-inj-r}} \vol(Z)^{\ref{k:ss-exp-vol}},
\]
two elements $h_1, h_2 \in  \Omega_{\tlg,\eta}$, and some $\gamma\in\tlg(\Q)$ so that 
$
\tilde g \gamma=h_1\tilde hh_2.
$
If $q\not\in S$, then $({\pi(\tilde g}\gamma))_q=\pi(\gamma)_q\in\SL_d(\Z_q)$. 
The claim thus follows with $\delta=\gamma$ (thought of as an element in $\Delta$).
\end{proof}

\subsection{The adjoint action}\label{sec:adj-L}\label{sec:simply-conn}
We now turn to a version of Theorem~\ref{thm:eff-strong-app-direct} where $\height_S(g)$
is replaced by a height function defined using the adjoint representation of $L$ on $\mathfrak l$.

First, we need some more notation. 
For all $v\in \places$, let $\|\;\|_{v}$ denote the maximum norm on $\fsl_d(\Q_v)$ with respect to the standard basis.
Using this family of norms, we define $\height(\lcagr)$ analogously to what was done in \S\ref{sec:height}.

Fix a $\Z$-basis $\mathcal B=\{v_1,\ldots,v_N\}$ for $\Lie(\lcagr)\cap \fsl_d(\Z)$
with $\|v_i\|_\infty \ll \height(\lcagr)^\star$. 
Using this basis, we identify $\Lie(\lcagr)\cap \fsl_d(\Z)$ with $\Z^N$ and $\Lie(\lcagr)$ with $\Q^N$;
in this way, $\SL(\Lie(\lcagr))$ is identified with $\SL_N$.
We also let $\|\;\|_{\Bcal,v}$ denote the maximum norm with respect to $\Bcal$ on $\Lie(\lcagr)(\bbq_v)$. 
To avoid confusion, we will keep the index ${}_\Bcal$ for functions defined using these norms, e.g.~we write $\cfb$ and $\height_{\Bcal}$ (although after the above identifications, they correspond precisely to the notions introduced in \S \ref{sec:alpha1}).

Let 
$\Ad_{\lcagr}:\lcagr\to\SL_N$ 
denote the adjoint representation. 
We sometimes write $\Ad_L$ or simply $\Ad$ for $\Ad_\lcagr$ if there is no confusion.  
Put $\cfs(w) :=\prod_S\|w_v\|_v$ for all $w=(w_v)\in\fl$.

Let $\fl(\Z_S):=\fl\cap \fsl_d(\Z_S)$;
note that $\fl(\Z_S)$ is invariant under the adjoint action of $\lcgr\cap \SL_d(\Z_S)$.
For every $g\in\lcgr$, we define  
\[
\hl(g):=\max\{\cfs(\Ad(g)w)^{-1}:0\neq w\in\fl(\Z_S)\}.
\]
The function $\hl$ is $\lcgr\cap \SL_d(\Z_S)$-invariant, so it defines a function on 
$L/\lcgr\cap \SL_d(\Z_S)$ which we continue to denote by $\hl$. 

As before, we put $|g|=\max\{\|g\|,\|g^{-1}\|\}$ for all $g\in\SL_N(\Q_v)$, where $\|\;\|$ is the operator norm on $\SL_N(\Q_v)$
with respect to some fixed norm on $\Q_v^N$, say the max norm with respect to the standard basis.

Let $\bR'=\Ad_\lcagr(\bR)$.
Put $\G=\tilde\bH \ltimes\bR'$, where the action of $\tilde\bH$ on $\bR'$ factors through the action of $\Ad_\lcagr (\bH)$ via 
$\Ad_\lcagr \circ \pi'$, where $\pi':\tilde\bH\to\bH$ is the natural covering map. 

The adjoint action on $\lcagr$ induces a homomorphism $\iota:\G\to\SL_N$ with finite central kernel, given by 
\[
\iota(g_{\tilde \bH} g_\bR') = \Ad_\lcagr (\pi'(g_{\tilde \bH})) g_\bR'.
\] 
In accordance to \S \ref{sec:data}, we set $Y:=\iota(\G(\A)/\G(\Q)) \subset \SL_N(\A) / \SL_N(\Q)$. Define $\hat Y$ as in \S\ref{sec:vol-vol} by replacing the pair $(\tlg,\pi)$ with $(\G,\iota)$ and $\SL_d$ by $\SL_N$; similarly fix an open subset $\Omega_S\subset\SL_N(\Q_S)$, and define $\vol(\hat Y)$ using $\Omega_S\subset\SL_N(\Q_S)$. We put 
\[
\text{$\height_\Bcal(Y)=\max\{\height(\lcagr),\vol(Y)\}\;$ and $\;\height_\Bcal(\hat Y)=\max\{\height(\lcagr),\vol(\hat Y)\}$}.
\]

Additionally, there is an epimorphism $\varphi: \tlg \to \bG$ given by $g_{\tilde \bH}g_{\bR} \mapsto g_{\tilde \bH} \Ad_\lcagr (g_{\bR})$, whose kernel is contained in $\rm Z(\bR)$, hence is unipotent. As was argued in \S\ref{sec:Levi-iota-G}, this implies that $\tlg(\Q)$ surjects onto $\G(\Q)$, and $\tlg(\Q_v)$ surjects onto $\G(\Q_v)$
for all $v \in \places$.

\begin{center} \begin{tikzcd}
\tlg \arrow[r, "\pi"] \arrow[d, "\varphi"] & \lcagr \arrow[d, "\Ad_\lcagr"] \\
\bG \arrow[r,"\iota"]& \SL_N
\end{tikzcd} \end{center}

As before, for every $g\in\tlg(\Q_S)$ we write $\tilde g=(g,(e)_{q\not\in S})\in\tlg(\A)$ and we write  
\be\label{eq:hat-g}
\hat g=\iota(\varphi(\tilde g))=(\Ad_\lcagr (\pi(g)),(e)_{p\not\in S})\in\iota(\G(\A)).
\ee
In what follows, the notation will confound the implicit diagonal embeddings of $\tlg(\Q)$ in $\tlg(\Q_S)$ and in $\tlg(\A)$. Which embedding is relevant will be indicated by the context.

\begin{lemma}\label{lem:ht-htL}
There exists some $\consta\label{k:ht-htL-htL}$ 
so that the following holds. For any $g \in L$ we have 
\[
\height(\lcagr)^{-\ref{k:ht-htL-htL}} \hl(g) \ll \height_{\Bcal}( (\Ad_\lcagr (g), (e)_{p \notin S})) \ll \height(\lcagr)^{\ref{k:ht-htL-htL}}\hl(g).
\]
\end{lemma}

\begin{proof}
For $g \in L$, set $\hat g :=  (\Ad_\lcagr (g), (e)_{p \notin S}) \in \SL_N(\A)$. For any $w \in \Q^N$, let $\bar w$ be a primitive integral vector on $\Q\cdot w$. First, observe that
\begin{align}
\notag \cfb( \hat g w)	&=\prod_\places\|\hat g_v w\|_{\Bcal, v}=\prod_\places\|\hat g_v \bar w\|_{\Bcal, v}			&&\text{by the product formula}\\
\notag			&=\prod_{S}\|\Ad(g)_v \bar w\|_{\Bcal, v} \cdot \prod_{p\not\in S}\| \bar w\|_{\Bcal, p}		\\
\notag			&= \prod_{S}\|\Ad(g)_v \bar w\|_{\Bcal,v}
					&&\text{since $\bar w$ is primitive integral}\\
\notag			&\gg \prod_{S}\|\Ad(g)_v \bar w\|_{v} \cdot \prod_{S} (\max_i \| v_i \|_v)^{-1} 				&&\| \; \|_{v} \ll (\max_i \| v_i \|_v) \cdot \| \; \|_{\Bcal,v} \\
\notag			&\geq \prod_{S}\|\Ad(g)_v \bar w\|_{v} \cdot (\max_i \| v_i \|_\infty)^{-1} 					&&\text{since $v_i \in \fsl_d(\Z)$} \\
\label{eq:cfb-cfs-again}&\gg \height(\lcagr)^{-\star} \cfs(\Ad(g) \bar w)
					&&\text{because $\|v_i\|_\infty \ll \height(\lcagr)^\star$}.
\end{align}
From this, it follows that
\begin{align}
\notag\height_{\Bcal}(\hat g)&=\max\{\cfb(\hat gw)^{-1}:0\neq w\in\Q^N\}&&\text{see~\eqref{eq:def-ht-g}}\\
\notag&\ll\height(\bL)^{\star}\max\{\cfs(\Ad (g) \bar w)^{-1}:0\neq w\in\Q^N\}&&\text{by~\eqref{eq:cfb-cfs-again}} \\
\notag&\leq \height(\bL)^{\star}\max\{\cfs(\Ad(g)w)^{-1}:0\neq w\in\Z_S^N\}\\
\notag&=\height(\bL)^{\star}\hl(g).
\end{align}

Similarly, since for every $w\in \Z_S^N$ and all $q \notin S$ we have $\|w\|_{\Bcal, q}\leq 1$, we get
\begin{align*}
\cfs(\Ad(g)w) 	&= \prod_{S} \| \Ad(g)_v w \|_v \\
			&\gg \height(\bL)^{-\star} \prod_{S} \| \Ad(g)_v w \|_{\Bcal, v}\\
			&\geq \height(\bL)^{-\star} \prod_{S} \| \Ad(g)_v w \|_{\Bcal, v} \prod_{q \notin S} \|w\|_q \\
			&= \height(\bL)^{-\star} \cfun_\Bcal(\hat g w).
\end{align*}
This implies the lower bound $\hl(g) \ll \height(\bL)^\star \height_\Bcal(\hat g)$. 
\end{proof}

\begin{thm}\label{thm:eff-strong-app}
There exists some $\consta\label{k:local-exp-z}\label{k:local-exp}$ so that the following holds. 
Let $\lcagr$ be any $\Q$-subgroup of $\SL_d$ with $\rad(\lcagr)=\rad_u(\lcagr)$ and let $\tilde \bL$, $(\bG, \iota)$, etc.~be as in \S \ref{sec:adj-L}. There exists some prime $p\ll \bigl(\log\disc_{\Bcal}(\hat Y)\bigr)^2$ with the following property. 
For any $g\in\tlg(\Q_S)$, there exists some $\gamma\in\tlg(\Q)$ so that $\iota(\varphi(\gamma))_q\in\SL_N(\Z_{q})$ for all $q\not\in S\cup\{p\}$ and
\[
|\iota(\varphi(g\gamma))_v|\ll \hl(\pi(g))^{\ref{k:local-exp}}\disc_{\Bcal}(\hat Y)^{\ref{k:local-exp}}
\]
for all $v\in S$.
Moreover, if $p\not\in S$, then 
\[
|\iota(\varphi(\gamma))_p|\ll \hl(\pi(g))^{\ref{k:local-exp}}\disc_{\Bcal}(\hat Y)^{\ref{k:local-exp}}.
\]
\end{thm}

\begin{proof}
In view of part~(1) of proposition in \S\ref{prop:vol-vol}, it suffices to prove the above estimates with $\disc_\Bcal(\hat Y)$ replaced by $\disc_\Bcal(Y)$.

Let $g \in \tilde \bL(\Q_S)$ and write $g=g_Hg_R$ where $g_H\in\tilde\bH(\Q_S)$ and $g_R\in\bR(\Q_S)$.

In virtue of \eqref{eq:str-cont}, we have that $\height_\Bcal (\ad(\Lie (\bL))) \ll \height(\bL)^\star$. Since $\Lie (\iota (\bG))$ $= \Lie (\Ad (\bL)) = \ad(\Lie (\bL))$, this means that $\height_\Bcal (\bG) \ll \height(\bL)^\star$.
Lemma~\ref{lem:height-g-g-H} thus yields
\be\label{eq:hatg-hatgH}
\height_{\Bcal}(\bL)^{-\star} \height_{\Bcal}(\hat g)^\star\ll\height_{\Bcal}(\hat g_H)\ll\height(\bL)^\star\height_{\Bcal}(\hat g)^\star.
\ee

As before, we write $Y_H=\iota(\tilde\bH(\A)/\tilde\bH(\Q))$.
Let $p\ll\bigl(\log\vol_\Bcal(Y_H)\bigr)^2$ be as in Theorem~\ref{sec:proof-semisimple} applied to $(\tilde\bH, \iota_{|_{\tilde \bH}})$, so that (combined with Lemma \ref{lem:ht-injr}) we have the following. 
There exists some $\gamma_0\in\tilde\bH(\Q)$ so that if we put $h' = (h'_S, h'_p, (h'_q)_{q \notin S \cup \{p\}}) :=\tilde g_H\gamma_0$, then $\iota(\varphi(h'))_q\in\SL_N(\Z_q)$ for all $q\not\in\{\infty, p\}$, $|\iota(\varphi(h'))_\infty|\ll_{\Bcal}1\ll\height(\lcagr)^\star$, and 
\begin{align}
\notag |\iota(\varphi(h'))_p| &\ll_{\Bcal}  \height_{\Bcal}(\hat g_H)^\star\vol_{\Bcal}(Y_H)^\star\\
\notag				&\ll \height(\lcagr)^\star \height_\Bcal(\hat g)^\star\vol_{\Bcal}(Y_H)^\star				&&\text{by~\eqref{eq:hatg-hatgH}}\\
\notag				&\ll\height(\lcagr)^\star\hl(\pi(g))^\star\vol_{\Bcal}(Y_H)^\star						&&\text{by Lemma~\ref{lem:ht-htL}}\\
\label{eq:hp-bd-local-semi}&\ll\hl(\pi(g))^\star\height_{\Bcal}(Y)^\star						&&\text{by Proposition~\ref{prop:vol-comp}}.
\end{align}
Also by Proposition~\ref{prop:vol-comp}, we have 
\be\label{eq:local-p-est}
p\ll \bigl(\log\vol(Y_H)\bigr)^2\ll \bigl(\log\height_{\Bcal}(Y)\bigr)^\star. 
\ee

Apply Lemma~\ref{lem:unip-case} with the set of places $\{\infty\}$ and $v=\infty$ to the element $\gamma_0^{-1}\tilde g_R\gamma_0$ to obtain some $\gamma_1\in\bR(\Q)$ such that 
\begin{itemize}
\item[(a)] $\pi(\gamma_0^{-1}\tilde g_R\gamma_0\gamma_1)\in \SL_d(\Z_q)$ for all primes $q$, and
\item[(b)] $|\pi((\gamma_0^{-1}\tilde g_R\gamma_0\gamma_1)_\infty)|\ll \height(\bR)^\star\ll \height(\lcagr)^\star$. 
\end{itemize}
Since $\pi(\gamma_0^{-1}\tilde g_R\gamma_0)_q=e$ for all $q\not\in S$, 
item (a) above implies that $\pi(\gamma_1)_q\in\SL_d(\Z_q)$ for all $q\not\in S$.

\medskip

Put $\gamma=\gamma_0\gamma_1\in\tlg(\Q)$ and write
\[
h= (h_S, h_p, (h_q)):=\tilde g\gamma=\tilde g_H\gamma_0(\gamma_0^{-1}\tilde g_R\gamma_0\gamma_1) = h'(\gamma_0^{-1}\tilde g_R\gamma_0\gamma_1). 
\]
The above estimates then imply that
\begin{enumerate}
\item By (a) and $\iota(\varphi(h_q'))\in \SL_N(\Z_q)$ we have $\iota(\varphi(h_q))\in\SL_N(\Z_q)$ for all $q \notin \{\infty, p\}$.  

\item By (b) and $|\iota(\varphi(h_\infty'))|\ll \height(\lcagr)^\star$ we have
\begin{align*}
|\iota(\varphi(h_\infty))| 	&\ll \height(\lcagr)^\star |\iota(\varphi((\gamma_0^{-1} \tilde g_R\gamma_0\gamma_1)_\infty))| = \height(\lcagr)^\star |\Ad(\pi((\gamma_0^{-1} \tilde g_R\gamma_0\gamma_1)_\infty))| \\
 					&\ll_{\Bcal} \height(\lcagr)^\star |\pi((\gamma_0^{-1} \tilde g_R\gamma_0\gamma_1)_\infty)| \ll \height(\lcagr)^\star.
\end{align*}

\item For the prime $p$ we have 
\begin{align*}
|\iota(\varphi(h_p))| 	&\leq |\iota(\varphi(h'_p))| \cdot |\Ad(\pi(\gamma_0^{-1}\tilde g_R\gamma_0\gamma_1))_p)| \\
				&\ll_{\Bcal} |\iota(\varphi(h'_p))| \cdot |\pi((\gamma_0^{-1}\tilde g_R\gamma_0\gamma_1)_p)| \\
				&\ll \height(\bL)^\star |\iota(\varphi(h'_p))|
					&&\text{by (a)}\\
				&\ll \hl(\pi(g))^\star\disc_{\Bcal}(Y)^\star
					&&\text{by~\eqref{eq:hp-bd-local-semi}.}
\end{align*}
\end{enumerate}

Let now $q\not\in S \cup \{p\}$. Then $\hat g_q = \iota(\varphi(\tilde g_q))=e$ and hence we have $\iota(\varphi(\gamma_q)) = \iota(\varphi(h_q)) \in\SL_N(\Z_q)$ by~(1). This means $\iota(\varphi(\gamma))\in\SL_N(\Z_{S\cup\{p\}})$. 

Lastly, if $p\not\in S$, we have again $\hat g_p=e$, therefore $\iota(\varphi(\gamma_p)) = \iota(\varphi(h_p))$ and (3) above gives the desired bound on $\iota(\varphi(\gamma_p))$. 
\end{proof}

The above proof actually gives the following stronger statement.

\begin{thm}\label{thm:eff-strong-app-revisited}
There exists some $\ref{k:local-exp-z}$ so that the following holds. 
Let $\lcagr$ be any $\Q$-subgroup of $\SL_d$ with $\rad(\lcagr)=\rad_u(\lcagr)$ and let $\tilde \bL$, $(\bG, \iota)$, etc.~be as in \S \ref{sec:adj-L}. There exists some prime $p\ll \bigl(\log\disc_{\Bcal}(\hat Y)\bigr)^2$ with the following property. 
Let $g\in\tlg(\Q_S)$ and write $g=g_Hg_R$ where $g_H\in\tilde\bH(\Q_S)$ and $g_R\in\bR(\Q_S)$.
There exists some $\gamma_0\in\tilde \bH(\Q)$ and some $\gamma_1\in \bR(\Q)$ with  
\begin{itemize}
\item $\iota(\varphi(\gamma_0))_q\in\SL_N(\Z_q)$ for all $q\not\in S\cup\{p\}$
\item $\pi(\gamma_1)_q\in\SL_d(\Z_q)$ for all $q\not\in S$
\item if $p\not\in S$, then $|\iota(\varphi(\gamma_0 \gamma_1))_p|\ll \hl(g)^{\ref{k:local-exp}}\disc_{\Bcal}(\hat Y)^{\ref{k:local-exp}}$,
\end{itemize}
so that if we write $(g,(e)_{\not\in S})\gamma_0\gamma_1=h_Hh_R$, where $h_H\in\tilde \bH(\A)$ and $h_R\in\bR(\A)$, then we have the following estimates.
\begin{enumerate}
\item $\pi(h_R)_q\in \SL_d(\Z_q)$ for all primes $q$,
\item $|\pi(h_R)_\infty|\ll  \height(\lcagr)^{\ref{k:local-exp-z}}$,
\item $\iota(\varphi(h_H))_q\in\SL_N(\Z_q)$ for all $q\not\in\{\infty, p\}$, 
\item $|\iota(\varphi(h_H))_\infty|\ll \height(\lcagr)^{\ref{k:local-exp-z}}$, and
\item $|\iota(\varphi(h_H))_p|\ll \hl(g)^{\ref{k:local-exp-z}}\disc_{\Bcal}(\hat Y)^{\ref{k:local-exp}}$.
\end{enumerate}
\end{thm} 

\medskip

For any $g\in \tilde{\bf H}(\Q_S)$, we define  
\[
\height_H(\pi(g)):=\max\{\cfs(\Ad_H(\pi(g))w)^{-1}:0\neq w\in\mathfrak h(\Z_S)\},
\]
where $\mathfrak h=\Lie(H)\cap\mathfrak{sl}_d(\Z_S)$. 

It follows from the definition that $\height_H(\pi(g))\leq \height_L(\pi(g))$ for any $g\in \tilde{\bf H}(\Q_S)$.
Moreover, in view of Lemma~\ref{lem:ht-htL} and Lemma~\ref{lem:height-g-g-H} we have the following. Let $g\in\tilde{\bf L}(\Q_S)$ and write 
$g=g_Hg_R$, then 
\be\label{eq:htL:htH-gH}
\height_H(\pi(g_H))\leq \height_L(\pi(g_H))\ll \height({\bf L})^\star\height_L(\pi(g))^\star.
\ee

We also record the following lemma.

\begin{lemma}\label{cor:htLg:htH-gH-lower}
Let $g\in\tilde{\bf L}(\Q_S)$ and write 
$g=g_Hg_R$, then 
\[
\height_H(\pi(g_H))\gg \disc_{\Bcal}(\hat Y)^{-\star} \height_L(\pi(g_H))^\star\gg \disc_{\Bcal}(\hat Y)^{-\star}\height_L(\pi(g))^\star.
\]
\end{lemma} 

\begin{proof}
The second estimate follows from Lemma~\ref{lem:ht-htL}, Lemma~\ref{lem:height-g-g-H}, and the fact that
$\disc_{\Bcal}(\hat Y)\geq \height({\bf L})$. Thus we only need to show 
\[
\height_H(\pi(g_H))\gg \disc_{\Bcal}(\hat Y)^{-\star} \height_L(\pi(g_H))^\star.
\]

The proof uses arguments similar to the ones used in the proof of Theorem~\ref{thm:eff-strong-app};
apply Theorem~\ref{thm:eff-strong-app-revisited} with ${\bf L}={\bf H}$ to $g_H$. 
There exist some $p\ll \bigl(\log\disc_{\Bcal}(\hat Y)\bigr)^2$ and $\gamma_0\in\tilde{\bf H}(\Q)$ so that 
\begin{itemize}
\item[(i)] $\Ad_H(\gamma_0)_q\in\SL_N(\Z_q)$ for all $q\not\in S\cup\{p\}$,
\item[(ii)] if $p\not\in S$, then $|\Ad_H(\gamma_0)_p|\ll \height_H(\pi(g_H))^{\ref{k:local-exp}}\disc_{\Bcal}(\hat Y)^{\ref{k:local-exp}}$, 
\end{itemize}
and if we put $h_H=(g_H,(e)_{q \not\in S})\gamma_0$, we have the following estimates.
\begin{enumerate}
\item $\Ad(h_H)_q\in\SL_N(\Z_q)$ for all $q\not\in\{\infty, p\}$, 
\item $|\Ad(h_H)_\infty|\ll \height(\lcagr)^{\ref{k:local-exp-z}}$, and
\item $|\Ad(h_H)_p|\ll \height_H(\pi(g_H))^{\ref{k:local-exp-z}}\disc_{\Bcal}(\hat Y)^{\ref{k:local-exp}}$.
\end{enumerate}

Apply Lemma~\ref{lem:unip-case} with the set of places $\{\infty\}$ and $v=\infty$ to the element $\gamma_0^{-1}\tilde g_R\gamma_0$ to obtain some $\gamma_1\in\bR(\Q)$ such that 
\begin{itemize}
\item[(a)] $\pi(\gamma_0^{-1}\tilde g_R\gamma_0\gamma_1)\in \SL_d(\Z_q)$ for all primes $q$, and
\item[(b)] $|\pi((\gamma_0^{-1}\tilde g_R\gamma_0\gamma_1)_\infty)|\ll \height(\bR)^\star\ll \height(\lcagr)^\star$. 
\end{itemize}
Since $\pi(\gamma_0^{-1}\tilde g_R\gamma_0)_q=e$ for all $q\not\in S$, 
item (a) above implies that $\pi(\gamma_1)_q\in\SL_d(\Z_q)$ for all $q\not\in S$.

Let us put $h_R=\gamma_0^{-1}\tilde g_R\gamma_0\gamma_1$, so that we have 
\[
(g, (e))\gamma_0\gamma_1=(g_Hg_R,(e))\gamma_0\gamma_1=((h_H)_S,(h_H)_{q \not\in S})((h_R)_S,(h_R)_{q \not\in S}). 
\]
By abuse, we denote the projection of $\gamma_0, \gamma_1$ onto the $S$-coordinates again by $\gamma_0, \gamma_1\in \tilde{\bf L}(\Q_S)$.

We have
\begin{align}
\notag\hl(g)&=\max\{\cfs(\Ad_L(\pi(g))w)^{-1}:0\neq w\in\mathfrak l(\Z_S)\}\\
\label{eq:hlg-hH-hR}&=\max\{\cfs(\Ad_L(\pi((h_H)_S(h_R)_S)\gamma_1^{-1}\gamma_0^{-1})w)^{-1}:0\neq w\in\mathfrak l(\Z_S)\}.
\end{align}
First, we note that using (1)--(3), (a) and (b) we have 
\begin{equation}\label{eq:hH-hR-bdd}
\cfs(\Ad_L(\pi((h_H)_S(h_R)_S)\gamma_1^{-1}\gamma_0^{-1})w)^{-1}\ll
\height_H(\pi(g_H))^{\star}\disc_{\Bcal}(\hat Y)^{\star} \cfs(\Ad_L(\gamma_1^{-1}\gamma_0^{-1})w)^{-1}.
\end{equation}
Furthermore, using (i), (ii), and the fact that $\pi(\gamma_1)_q\in\SL_d(\Z_q)$ for all $q\not\in S$, we have 
\[
\cfs(\Ad_L(\gamma_1^{-1}\gamma_0^{-1})w)^{-1}\ll \height_H(\pi(g_H))^{\star}\disc_{\Bcal}(\hat Y)^{\star}.
\]
This, in view of~\eqref{eq:hH-hR-bdd} and~\eqref{eq:hlg-hH-hR}, implies that
\[
\notag\hl(g)\ll \height_H(\pi(g_H))^{\star}\disc_{\Bcal}(\hat Y)^{\star};
\]
the proof is complete.
\end{proof}

\subsection{Uniform lattices}\label{sec:Mahler-Meas}
In this section, we discuss the dependence of the above estimates on $\hl(g)$ under the assumption that 
the Levi component, ${\bf H}$, of $\bf L$ is $\Q$-anisotropic.  We begin with the following lemma which is of independent interest --- one could obtain similar estimates using known results towards the Lehmer conjecture, but we provide a homemade argument. 

\begin{lem}
There exists some $0<\beta<1$ depending on $\dim {\bf L}$ with the following property.
Let $w\in\mathfrak l(\Z_S)$ and assume that there exists some $g\in L$ so that $\cfun_S(\Ad_L(g) w)\leq \beta$. 
Then $w$ is a nilpotent element.
\end{lem}

\begin{proof}
Let $\bar\sigma(w)$ be the product of all the {\em nonzero} eigenvalues of $w$;
if this product is empty, i.e.~if $w$ is nilpotent, put $\bar\sigma(w)=0$.
Note that $\bar\sigma(w)\in\bbq$ because $\bar\sigma(w)$ is invariant under the Galois group of the splitting field of $w$. 
Further, since $w\in\mathfrak l(\Z_S)$, the product formula implies that
either $\cfun_S(\bar\sigma(w))\geq 1$ or $\bar\sigma(w)=0$. (Here, we also use $\cfun_S$ to denote the function $\Q_S \to \R^+: r \mapsto \prod_{v \in S} |r|_v$.) 

Let $\beta>0$ and assume that $\cfun_S(\Ad_L(g)w)\leq \beta$ for some $g\in L$. 
There exist some $r\in\Z_S^\times$
so that $\|r\Ad_L(g)w\|_v\asymp \cfun_S(\Ad_L(g)w)^\star$ for all $v\in S$, see for example \cite[Lemma 8.6]{KT:Nondiv}.
Therefore, all the eigenvalues of $r\Ad_L(g)w$ have $v$-norm $\ll \beta^\star$ for all $v\in S$.

Since $\cfun_S(r)=1$ and $\Ad_L(g)w$ has the same eigenvalues as $w$, we deduce that $\cfun_S(\bar\sigma(w))\geq 1$ cannot hold when $\beta$ is small enough; thus, $w$ is nilpotent.   
\end{proof}

\begin{prop}
Let the notation be as above; in particular, recall the Levi decomposition $\tilde{\bf L}=\tilde{\bf H}{\bf R}$ fixed in \S\ref{sec:local-notation}. 
Assume that $\tilde{\bf H}$ is $\Q$-anisotropic. Let $g\in \tilde{\bf L}(\Q_S)$, then
\[
\height_L(\pi(g))\ll \disc_{\Bcal}(\hat Y)^\star.
\]
Moreover, if $\tilde{\bf L}$ is semisimple, i.e.\ $\tilde{\bf L}=\tilde{\bf H}$, and we assume that $\tilde{\bf L}$ is $\Q$-anisotropic,  then 
$
\height_L(\pi(g))\ll 1.
$
\end{prop}

\begin{proof}
Let us write $g=g_Hg_R$ where $g_H\in\tilde{\bf H}(\Q_S)$ and $g_R\in{\bf R}(\Q_S)$.
Let $\beta$ be as in the previous lemma applied with $\bH$ instead of $\bL$. 
We claim that $\height_H(\pi(g_H))\leq \beta^{-1}$. Indeed, if $\height_H(\pi(g_H))> \beta^{-1}$, then by definition there exists a nonzero $w \in \fh(\Z_S)$, such that $\cfun_S(\Ad_H(\pi(g_H)) w) < \beta$. The lemma then implies that $w$ is a nilpotent element. 
Exponentiating $w$, we get that $\bf H$ (and hence $\tilde{\bf H}$) is $\Q$-isotropic, which is a contradiction.  
This implies the proposition when $\tilde{\bf L}=\tilde{\bf H}$.

Now, for the general case, we apply Lemma~\ref{cor:htLg:htH-gH-lower} and the bound we obtained above to obtain
\[
\height_L(\pi(g))\ll \height_H(\pi(g_H))^{\star}\disc_{\Bcal}(\hat Y)^\star\ll\disc_{\Bcal}(\hat Y)^\star,
\]
as was claimed.
\end{proof}

It is worth mentioning that the proof of the previous proposition when $\tilde{\bf L}$ is semisimple is independent of 
Lemma~\ref{cor:htLg:htH-gH-lower} and relies only on the lemma proved in this section.

\bibliographystyle{plain}

\end{document}